\documentclass[psamsfonts]{amsart}

\usepackage{amssymb, amsthm, amsmath, amsfonts, mathrsfs}
\usepackage{enumerate}
\usepackage[all,arc]{xy}
\usepackage{hyperref}
\usepackage{graphicx}
\usepackage{overpic}
\usepackage{xcolor}
\usepackage{mathtools}

\newtheorem{thm}{Theorem}[section]
\newtheorem{cor}[thm]{Corollary}
\newtheorem{prop}[thm]{Proposition}
\newtheorem{lem}[thm]{Lemma}
\newtheorem{conj}[thm]{Conjecture}

\theoremstyle{definition}
\newtheorem{defn}[thm]{Definition}

\newtheorem{obs}[thm]{Observation}

\theoremstyle{remark}
\newtheorem{rem}[thm]{Remark}
\newtheorem{rems}[thm]{Remarks}

\makeatletter
\let\c@equation\c@thm
\makeatother
\numberwithin{equation}{section}

\definecolor{orange}{hsb}{0.067,1,1}
\definecolor{lime}{hsb}{0.33,1,1}
\definecolor{purple}{hsb}{0.83,1,0.5}

\title{L-space surgeries on $(1,1)$-knots in $S^1\times S^2$}

\author{Qingfeng Lyu}
\address{Department of Mathematics, Boston College\\ Chestnut Hill, MA 02467}
\email{lyuqi@bc.edu}

\author{Zipei Nie}
\address{Department of Mathematics, University of Illinois at Urbana–Champaign, Urbana,
IL 61801, USA}
\email{znie@illinois.edu}

\date{\today}

\begin{document}

\begin{abstract}
    We extend the diagrammatic criterion for $(1,1)$ L-space knots in~\cite{greene2018space} to $(1,1)$-knots in $S^1\times S^2$. We also discuss applications to L-space surgeries on two-bridge links.
\end{abstract}

\maketitle

\section{Introduction}
\label{sec:intro}

A rational homology $3$-sphere $Y$ is called an L-space if it has minimal Heegaard Floer homology, that is, $\mathrm{rank}\;\widehat{HF}(Y)=|\mathrm{H}_1(Y;\mathbb{Z})|$. L-spaces are Heegaard Floer-theoretic generalizations of lens spaces, and their topological characterization remains an important open problem. This problem is closely related to the L-space conjecture~\cite{boyer2013spaces,juhasz2015survey}.

A knot $K$ in an L-space $Y$ is called an L-space knot if it admits a non-trivial L-space surgery. The L-space surgery slopes of such a knot form a computable interval~\cite{rasmussen2017floer}, and it is also of interest to find topological characterizations of L-space knots. Such characterizations have been established for $(1,1)$-knots in $S^3$ and lens spaces by Greene, Lewallen, and Vafaee~\cite{greene2018space} using $(1,1)$-diagrams.

A $(1,1)$-knot $K\subset Y$ is a knot that has genus-one bridge number one. Its $(1,1)$ position can be described by a $(1,1)$-diagram $(\Sigma,\alpha,\beta,z,w)$, where $(\Sigma,\alpha,\beta)$ is a genus-one Heegaard diagram for $Y$, and $z,w$ are two basepoints on the Heegaard torus $\Sigma$. The knot $K$ can be obtained by taking boundary-parallel arcs in the two solid tori that connect the two basepoints and avoid the compression disks bounded by $\alpha$ and $\beta$.

\begin{figure}[!hbt]
    \centering
    \includegraphics[width=0.5\linewidth]{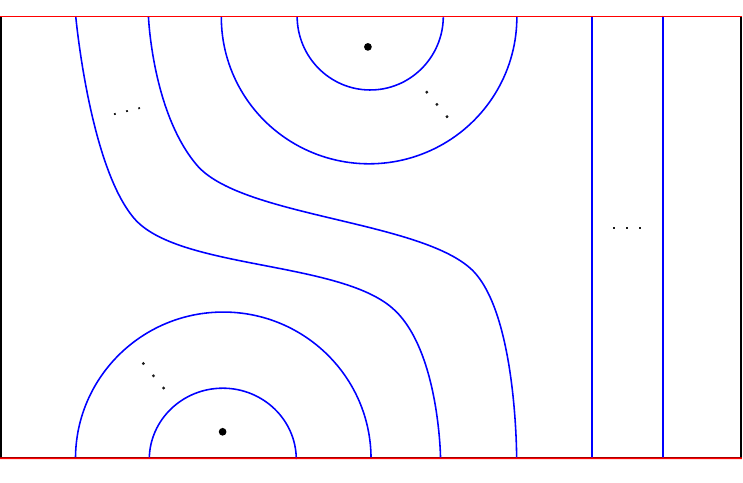}
    \put(-4,0){\color{red}$\alpha$}
    \put(-17,40){\color{blue}$\beta$}
    \put(-123,10){$z$}
    \put(-87,103){$w$}
    \caption{A reduced $(1,1)$-diagram}
    \label{fig:reduced11}
\end{figure}
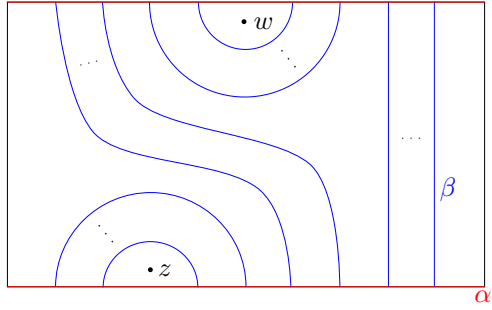

Given a $(1,1)$-diagram $(\Sigma,\alpha,\beta,z,w)$, one can always isotope the two curves $\alpha,\beta$ on the two-basepoint torus $(\Sigma,z,w)$ so that every bigon contains a basepoint. Such a $(1,1)$-diagram is called \textit{reduced}. A reduced $(1,1)$-diagram has a standard form as in Figure~\ref{fig:reduced11}. Notice that $\alpha$ cuts $\beta$ into several arcs. We call the $\beta$-arcs attached to the same side of $\alpha$ \textit{rainbow arcs}, and the $\beta$-arcs connecting different sides of $\alpha$ \textit{vertical arcs}.

\begin{defn}[\cite{greene2018space}, Definition 1.1]
    A reduced $(1,1)$-diagram $(\Sigma,\alpha,\beta,z,w)$ is called coherent if there exist orientations of $\alpha,\beta$ that induce coherent orientations on the boundary of every embedded bigon $(D,\partial D)\subset (\Sigma,\alpha\cup\beta)$.
\end{defn}

\begin{rem}
    If a reduced $(1,1)$-diagram is put in standard form as in Figure~\ref{fig:reduced11}, then it is coherent if and only if an orientation of $\beta$ induces coherent orientations for all the rainbow $\beta$-arcs.
\end{rem}

\begin{thm}[\cite{greene2018space}, Theorem 1.2]
    A reduced $(1,1)$-diagram in $S^3$ or a lens space represents an L-space knot if and only if it is coherent.
    \label{thm:GLV}
\end{thm}

In this paper, we extend the above result to $(1,1)$-knots in $S^1\times S^2$. 

\begin{thm}
    A reduced $(1,1)$-diagram in $S^1\times S^2$ represents a knot admitting L-space surgeries if and only if it is coherent, with the only exception being the local unknot in $S^1\times S^2$, whose reduced $(1,1)$-diagram has no bigons.
    \label{thm:S1S2}
\end{thm}

Although the knot Floer homology techniques used in~\cite{greene2018space} do not directly generalize to knots in $S^1\times S^2$, the diagrammatic criterion itself does not essentially rely on the ambient space being a rational homology sphere. Instead of generalizing those Floer homology techniques, we extend the methods developed in \cite{nie2021explicit,lyu2024knot,lyu2025persistent} to $(1,1)$-knots in $S^1\times S^2$.

$(1,1)$-knots in $S^1\times S^2$ arise naturally in the study of link surgeries in $S^3$. Examples include the links considered in~\cite{clay2025order,santoro2026space}. As an application of our result, we discuss L-space surgeries on $2$-bridge links and recover some results in~\cite{santoro2026space}.

In Section~\ref{sec:coherent} we show that a $(1,1)$-knot in $S^1\times S^2$ represented by a coherent, reduced $(1,1)$-diagram admits non-trivial L-space surgeries, except for the local unknot. In Section~\ref{sec:incoherent} we show that a $(1,1)$-knot in $S^1\times S^2$ represented by an incoherent, reduced $(1,1)$-diagram is persistently foliar, thus cannot admit L-space surgeries. Theorem~\ref{thm:S1S2} is proved at the end of Section~\ref{sec:incoherent}. In Section~\ref{sec:2bridge} we discuss applications to 2-bridge links.

\subsubsection*{Acknowledgements} We thank Jacob Rasmussen and Josh Greene for helpful conversations, and Yi Ni for suggesting the problem.  

\section{Coherent diagrams give L-space knots}
\label{sec:coherent}

In the standard terminology, an L-space is a rational homology sphere with minimal Heegaard Floer homology, and an L-space knot is a knot in an L-space that admits a nontrivial L-space surgery. We extend the definition of L-space knots as follows:

\begin{defn}
    Let $Y$ be a closed, oriented, and connected $3$-manifold. A knot $K\subset Y$ is called an L-space knot if $Y-K$ is Floer simple in the sense of~\cite{rasmussen2017floer}, that is, if it admits multiple Dehn fillings to L-spaces.
    \label{defn:L-sp_knot}
\end{defn}

\begin{rem}
Under this definition, being an L-space knot depends only on the knot complement and is therefore preserved under surgery duality.
\end{rem}

A special property of L-space knots in $S^1\times S^2$ is proved in~\cite{rasmussen2017floer}:

\begin{prop}[\cite{rasmussen2017floer}, Proposition 7.8]
    Let $Z$ be a rational homology sphere. Suppose $K\subset Z\# (S^1\times S^2)$ has an L-space surgery. Then the complement of $K$ is a ``generalized solid torus''. In particular, all non-longitudinal fillings of the knot complement yield L-spaces.
    \label{prop:g_torus}
\end{prop}

In this section, we prove the following proposition:

\begin{prop}
    Let $K\subset S^1\times S^2$ be a $(1,1)$-knot. If $K$ is represented by a reduced, coherent diagram, and $K$ is not the local unknot, then $K$ is an L-space knot. In fact, every non-trivial surgery along $K$ yields an L-space.
    \label{prop:coherent}
\end{prop}

\subsection{The 3-geodesics description}

Knots with L-space surgeries in $S^1\times S^2$ have been studied, for example, in~\cite{baker2016some,rasmussen2017floer,ni2019null}. In particular, we know from~\cite{ni2019null} that an L-space knot in $S^1\times S^2$ must be a spherical braid. It is therefore reasonable to first check that coherent diagrams in $S^1\times S^2$ give spherical braids. To that end, we apply the 3-geodesics description for $(1,1)$-knots with coherent diagrams, which was introduced in~\cite{nie2021explicit}:

\begin{prop}[cf.~\cite{nie2021explicit}, Theorem 1]
    Let $Y=U_0\cup_{\Sigma}U_1$ be a genus-one Heegaard splitting. Endow the Heegaard torus $\Sigma$ with the standard Euclidean geometry. A $(1,1)$-knot $K\subset Y$ is represented by a coherent $(1,1)$-diagram if and only if it is isotopic to a union of 3 arcs $K=\rho\cup\tau_0\cup\tau_1$, such that
    \begin{enumerate}[(1)]
        \item $\rho$ is a geodesic of $\Sigma$;
        \item $\tau_0$ is properly embedded in some meridional disk of $U_0$, whose boundary is a geodesic of $\Sigma$;
        \item $\tau_1$ is properly embedded in some meridional disk of $U_1$, whose boundary is a geodesic of $\Sigma$.
    \end{enumerate}
    \label{prop:geodesics}
\end{prop}

\begin{cor}
    Let $K\subset S^1\times S^2$ be a $(1,1)$-knot represented by a coherent $(1,1)$-diagram. If $K$ is not the local unknot, then it is a spherical braid in $S^1\times S^2$.
    \label{cor:braid}
\end{cor}

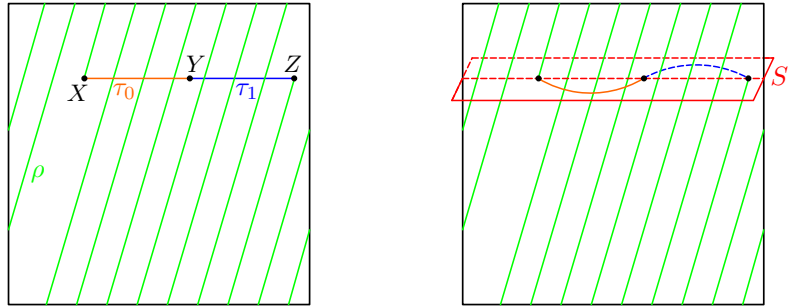
\begin{figure}[!hbt]
    \begin{overpic}[scale=0.5]{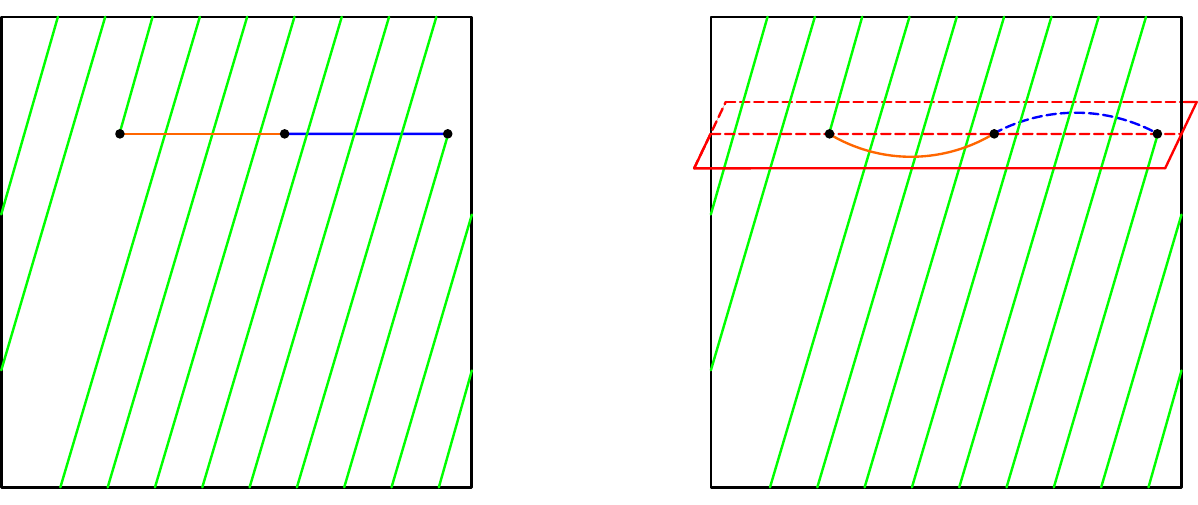}
        \put(3,18){\color{lime} $\rho$}
        \put(13.7,28.8){\color{orange} $\tau_0$}
        \put(29.7,28.8){\color{blue} $\tau_1$}
        \put(99.5,30){\color{red} $S$}
        \put(7.7,28){\small $X$}
        \put(23.3,31.7){\small $Y$}
        \put(36,31.7){\small $Z$}
    \end{overpic}
    \caption{The 3-geodesics description for a knot in $S^1\times S^2$}
    \label{fig:3_geod}
\end{figure}

\begin{proof}
    By Proposition~\ref{prop:geodesics}, $K$ admits a 3-geodesics description $K=\rho\cup\tau_0\cup\tau_1$; see Figure~\ref{fig:3_geod} left. Since $K$ lies in $S^1\times S^2$, the two meridional disks containing $\tau_0$ and $\tau_1$ share the same boundary and patch together to form a standard non-separating sphere $S=\{\star\}\times S^2 \subset S^1\times S^2$; see Figure~\ref{fig:3_geod} right. If $\rho$ has the same slope as $\tau_0$ and $\tau_1$, then $K$ lives in this 2-sphere $S$ and is therefore the local unknot. Otherwise $\rho$ is transverse to the $S^2$-fibers of $S^1\times S^2$, and in particular, intersects $S$ coherently. We can then slightly perturb the arc $\tau=\tau_0\cup\tau_1\subset S$ to make $K$ transverse to the $S^2$-fibers. It follows that $K$ is a spherical braid.
\end{proof}

\subsection{L-space surgeries}

For (spherical) braids,~\cite{rasmussen2017floer} proposed the following criterion for detecting L-space surgeries:

\begin{prop}[\cite{rasmussen2017floer}, Proposition 7.13]
    Let $\sigma$ be a $g$-strand braid in $S^1\times D^2$. Let $\bar{\sigma}$ be the ordinary braid closure in $S^3$ obtained by filling $S^1\times D^2$ along $S^1\times \{\star\}$, and let $\tilde{\sigma}$ be the spherical braid in $S^1\times S^2$ obtained by filling $S^1\times D^2$ along $\partial D^2$. Let $\Delta_g\in \mathrm{Br}_g$ be the $g$-strand full twist. Suppose $K_n=\overline{\Delta_g^n\sigma}$ is an L-space knot in $S^3$ for all $n\geq 0$. Then the complement of $\tilde{\sigma}$ is a generalized solid torus.
    \label{prop:7.13}
\end{prop}

We are now ready to prove Proposition~\ref{prop:coherent}:

\begin{proof}[Proof of Proposition~\ref{prop:coherent}]

    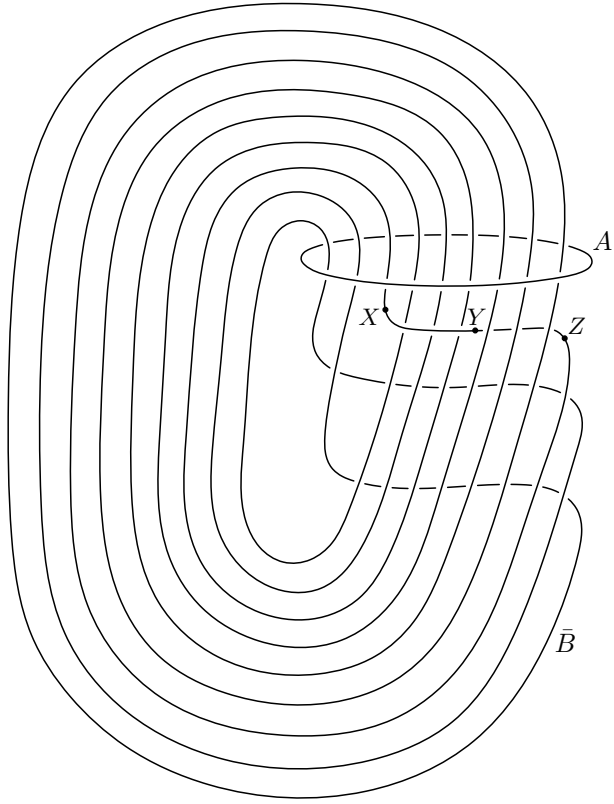
\begin{figure}[!hbt]
        \begin{overpic}[scale=0.5]{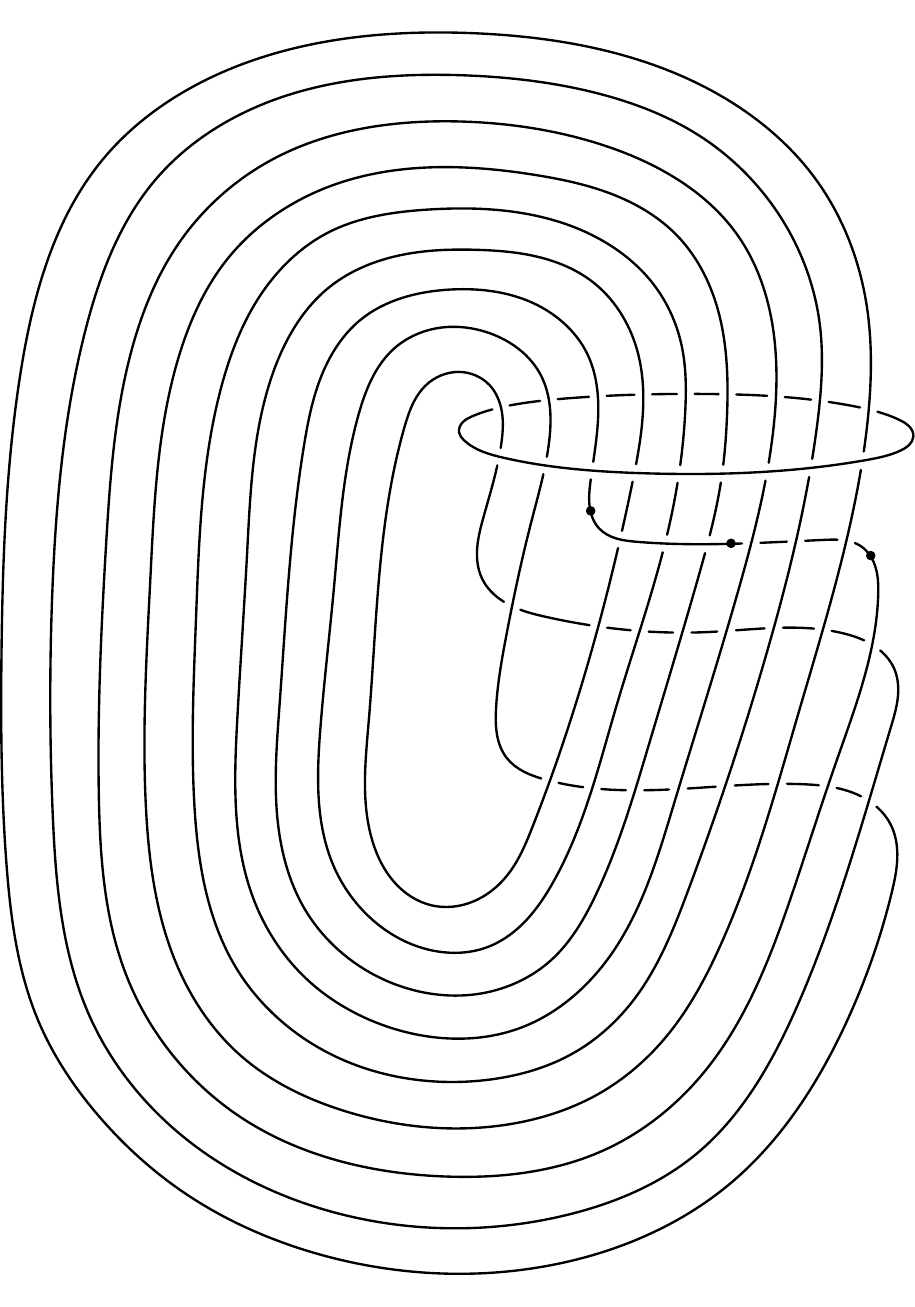}
            \put(70,68){$A$}
            \put(65.5,20){$\bar{B}$}
            \put(42,59){\small $X$}
            \put(55,59){\small $Y$}
            \put(67,58){\small $Z$}
        \end{overpic}
        \caption{The corresponding link $L=A\cup\bar{B}$ in $S^3$}
        \label{fig:link}
    \end{figure}
    
    According to Corollary~\ref{cor:braid}, $K$ is a spherical braid. Moreover, we can push $K$ into the interior of the solid torus $U_1$ (as in Proposition~\ref{prop:geodesics}) and consider the resulting braid $B\subset U_1$. The complement $U_1-\mathrm{int}(N(B))$ can then be understood as the complement of some link $L=A\cup\bar{B}$ in $S^3$, where $K$ can be recovered by performing $0$-surgery on $A$. See Figure~\ref{fig:link}, where the link corresponds to the knot described in Figure~\ref{fig:3_geod}.

    Now consider what happens if we slightly perturb the point $X$ in the 3-geodesics description along the direction of $\rho$. Suppose $X$ is perturbed to $X'$ so that the slope of $\tau_0'=X'Y$ is some rational number $r$ (for $|r|$ small), see Figure~\ref{fig:new_knot}. The new 3-geodesics description $(\rho',\tau_0',\tau_1)$ still describes a $(1,1)$-knot, but this new knot (which we denote as $K_r$) no longer lives in $S^1\times S^2$. In particular, if we pick $r=-\frac{1}{n}$ (for $n$ large enough), then we get a $(1,1)$-knot in $S^3$.

    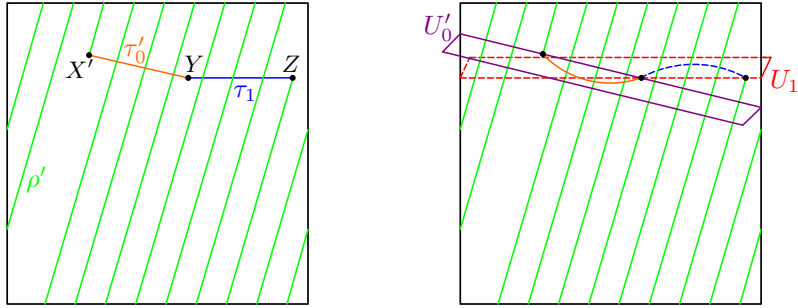
\begin{figure}[!hbt]
        \begin{overpic}[scale=0.5]{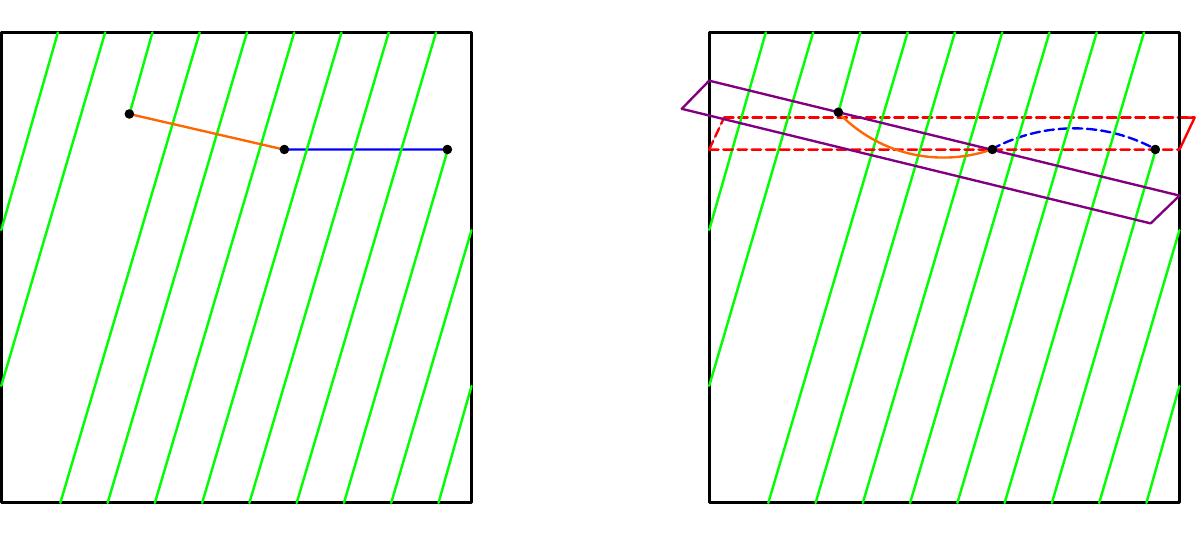}
            \put(2.6,18){\color{lime} $\rho'$}
            \put(15.3,35.1){\color{orange} $\tau_0'$}
            \put(29.7,29.8){\color{blue} $\tau_1$}
            \put(7.7,32.2){\small $X'$}
            \put(23.3,33){\small $Y$}
            \put(36,33){\small $Z$}
            \put(99.8,31){\color{red} $U_1$}
            \put(54.5,38){\color{purple} $U_0'$}
        \end{overpic}
        \caption{A new $(1,1)$-knot obtained by perturbing the endpoint $X$}
        \label{fig:new_knot}
    \end{figure}

    According to~\cite{nie2021explicit}, $K_{\left(-\frac{1}{n}\right)}$ is an L-space knot in $S^3$. Moreover, we can still push $K_{\left(-\frac{1}{n}\right)}$ into the solid torus $U_1$, and (for $n$ large) we will get the same braid $B$ in $U_1$ as before. As a consequence, $K_{\left(-\frac{1}{n}\right)}$ can be obtained from $L$ by performing $\left(-\frac{1}{n}\right)$-surgery on $A$.

    Now we apply Proposition~\ref{prop:7.13}. Suppose $B$ has $g$ strands, and let $\Delta_g$ denote the full twist for $g$-strand braids. Then $K_{\left(-\frac{1}{n}\right)}$ is essentially $\overline{\Delta_g^nB}$, which is an L-space knot in $S^3$ by the previous paragraph. Applying Proposition~\ref{prop:7.13} to $\overline{\Delta_g^NB}$ (for some large $N$), we know the complement of $K$ is a generalized solid torus. In particular, all non-trivial surgeries on $K$ yield L-spaces.
\end{proof}

\section{Incoherent diagrams give persistently foliar knots}
\label{sec:incoherent}

In this section, we prove the following proposition:

\begin{prop}
    Let $K\subset S^1\times S^2$ be a $(1,1)$-knot. If $K$ is represented by a reduced, incoherent diagram, then $K$ is persistently foliar. That is, for every boundary slope other than the knot meridian, there exists a co-oriented taut foliation of the knot complement $(S^1\times S^2 - K)$, which intersects the boundary torus transversely in circles of that slope.
    \label{prop:incoherent}
\end{prop}

\begin{rem}
    Since the taut foliation on $S^1\times S^2$ is unique up to isotopy, the knot complement admits a co-oriented taut foliation that is transverse to the boundary torus in circles of the knot meridian slope if and only if the knot is a spherical braid. It is unclear to us when this happens given a reduced, incoherent $(1,1)$-diagram representing the knot.
\end{rem}

Most of the arguments in~\cite{lyu2024knot,lyu2025persistent} carry over to $S^1\times S^2$, although a few points require slightly different or more careful treatment. We sketch the proof and highlight the differences below, referring the reader to~\cite{lyu2025persistent} for further details.

\subsection{Proof sketch}

We use branched surfaces to construct taut foliations. Starting from the $(1,1)$-diagram, we construct a branched surface, and when the diagram is incoherent, we apply certain modifications to it. We refer the reader to~\cite{lyu2025persistent}, Section~3 for the definitions of the ``modified Heegaard branched surfaces'' and recall the following proposition:

\begin{prop}[\cite{lyu2025persistent}, Proposition 3.5]
    Let $(\Sigma,\alpha,\beta,z,w)$ be an incoherent reduced $(1,1)$-diagram representing some $(1,1)$-knot $K$. Let $\mathcal{B}$ be a (type I or II) modified Heegaard branched surface of the $(1,1)$-diagram. If $\mathcal{B}$ fully carries a lamination, then $K$ is persistently foliar.
  \label{prop:lamitofoli}
\end{prop}

We use an inductive argument to construct laminations fully carried by these ``modified'' branched surfaces, where the key observation is a reduction operation on (reduced) $(1,1)$-diagrams. Let $(\Sigma,\alpha,\beta,z,w)$ be a reduced, non-simple $(1,1)$-diagram (i.e. the reduced diagram contains bigons). We can then define the ``carrying curve'' $\gamma$ of $\alpha$, such that $(\Sigma,\gamma,\beta,z,w)$ is a (strictly) simpler reduced $(1,1)$-diagram. Moreover, up to some further modifications, the branched surface $\mathcal{C}$ constructed for this simpler $(1,1)$-diagram $(\Sigma,\gamma,\beta,z,w)$ can be regarded as a sub-branched surface of the branched surface $\mathcal{B}$ constructed for the original $(1,1)$-diagram $(\Sigma,\alpha,\beta,z,w)$. We can then apply the ``blow up'' techniques summarized in~\cite{lyu2025persistent}:

\begin{defn}[\cite{lyu2025persistent}, Definition 6.18]
  Let $\mathcal{B}$ be a branched surface possibly with circle boundary. A branched surface $\mathcal{C}\subset \mathcal{B}$ is called a \textit{sub-branched surface} of $\mathcal{B}$ if $\partial \mathcal{C}\subset \partial \mathcal{B}$, and each branch sector of $\mathcal{C}$ is a union of branch sectors of $\mathcal{B}$.

  Let $\mathcal{A'}$ be the setwise difference $\mathcal{A'}=\mathcal{B}-\mathcal{C}$. Let $\mathcal{A}$ be the branched surface obtained by first attaching $\mathcal{A'}$ to the horizontal boundary of $N(\mathcal{C})$ (as if attached to $\mathcal{C}$), and then collapsing the boundary train tracks to circles. We call $\mathcal{A}$ the \textit{blow up branched surface of $\mathcal{B}$ with respect to $\mathcal{C}$}.
  \label{defn:blowup}
\end{defn}

\begin{prop}[\cite{lyu2025persistent}, Proposition 6.19]
  Let $\mathcal{B}$ be a branched surface possibly with circle boundary, and $\mathcal{C}\subset \mathcal{B}$ a sub-branched surface. Let $\mathcal{A}$ be the blow up of $\mathcal{B}$ with respect to $\mathcal{C}$. Suppose $\mathcal{C}$ fully carries a lamination and has no disk of contact. Suppose moreover that $\mathcal{A}$ fully carries a lamination. Then $\mathcal{B}$ fully carries a lamination.
  \label{prop:blowup}
\end{prop}

Once we prove that the blow up branched surface $\mathcal{A}$ always fully carries a lamination, Proposition~\ref{prop:blowup} allows us to reduce everything to what we call the ``monotone'' diagrams and ``primitive'' diagrams, where we find laminations directly using laminar branched surfaces.

\vspace{3pt}

\begin{rems}~\
    \begin{enumerate}
        \item In~\cite{lyu2025persistent} reductions occur in both sections 4 and 6. In Section~4, $\mathcal{A}$ is simple and carries a lamination trivially (in fact, $\mathcal{C}$ can usually be obtained by splitting $\mathcal{B}$, and a lamination fully carried by $\mathcal{C}$ passes directly to $\mathcal{B}$). In Section~6, $\mathcal{A}$ is more complicated and needs to be split to become laminar. 
        \item In Subsection~6.4 of~\cite{lyu2025persistent}, the modified branched surface is first reduced to an earlier branched surface constructed in~\cite{lyu2024knot}. Further reductions may then be performed as in~\cite{lyu2024knot}. Thus, we need to carry over the reductions from both papers to $S^1\times S^2$.
    \end{enumerate}
\end{rems}

Since the above arguments are mainly based on non-simple or incoherent diagrams, most of them also work for incoherent diagrams in $S^1\times S^2$. There are two main differences. First, a reduced $(1,1)$-diagram in $S^1\times S^2$ may have no vertical arcs. This slightly affects the diagram decomposition used to define the carrying curves and reduction operations. We address this in Subsection~\ref{subsec:diadecomp}. Second, at various stages we use the fact that the ambient space contains no non-separating closed surfaces to satisfy the conditions of the laminar branched surface theory. We explain
the necessary adjustments in Subsection~\ref{subsec:incomp}.

\subsection{Diagram decompositions}
\label{subsec:diadecomp}

In this subsection, we discuss differences in diagram decompositions in $S^1\times S^2$.

\vspace{6pt}

\textbf{Simple diagrams.} We still call a reduced $(1,1)$-diagram \textit{simple} if there are no bigons (or equivalently, no rainbow arcs). When the ambient space is $S^1\times S^2$, this also implies that there are no vertical arcs. It follows that $\alpha,\beta$ are two parallel curves that do not intersect. Depending on whether the two basepoints fall into the same or different annuli, the diagram represents the local unknot or the core knot in $S^1\times S^2$. These are in place of the ``simple knots'' in $S^3$ and lens spaces.

\vspace{6pt}

\textbf{No vertical arcs.} Since in $S^1\times S^2$ the algebraic intersection number of $\alpha$ and $\beta$ is $0$, there may be no vertical arcs when we place $\alpha$ in the standard horizontal position. When this happens, $\Sigma-(\alpha\cup\beta)$ consists of $2$ bigons, many quadrilaterals, and an annulus.

Choose arbitrary orientations of the curves $\alpha,\beta$ and recall the following definition:

\begin{defn}[\cite{lyu2024knot}, Definition 2.3]
  Fixing orientations of $(\Sigma,\alpha,\beta)$, for each $\alpha$-arc cut out by $\beta$, we say it is ($\beta$-)\textbf{sink} if it always lies to the left of the $\beta$ curve at its endpoints, ($\beta$-) \textbf{source} if it always lies to the right of the $\beta$ curve at endpoints, and ($\beta$-)\textbf{parallel} otherwise. See Figure~\ref{fig:sink_source_parallel} $a\sim c$. It follows immediately that for a quadrilateral region its two boundary $\alpha$-arcs must be of the same type. We call it a ($\beta$-)\textbf{sink sector} if its boundary $\alpha$-arcs are sink, a ($\beta$-)\textbf{source sector} if its boundary $\alpha$-arcs are source, and a ($\beta$-)\textbf{parallel sector} otherwise. See Figure~\ref{fig:sink_source_parallel} $d\sim f$. In addition, we call a bigon region a ($\beta$-)\textbf{sink bigon} if its boundary $\alpha$-arc is sink, and a ($\beta$-)\textbf{source bigon} if its boundary $\alpha$-arc is source.

  By interchanging $\alpha$ and $\beta$ in the above definition, we can also define $\alpha$-sink (source, parallel) arcs, $\alpha$-sink (source, parallel) sectors, and $\alpha$-sink (source) bigons.
  \label{def:sink_source_parallel}
\end{defn}

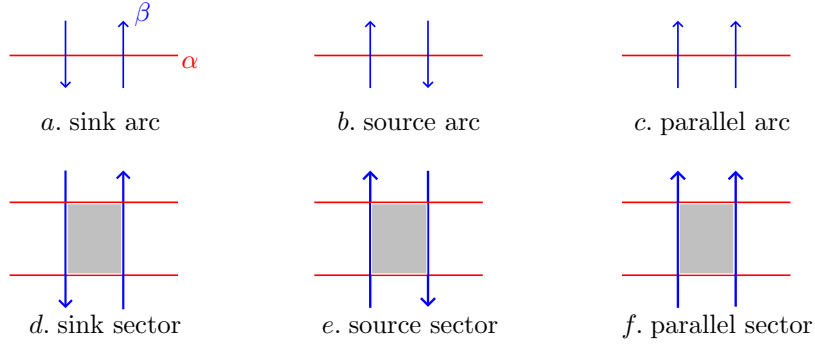
\begin{figure}[!hbt]
  \begin{overpic}[scale=0.6]{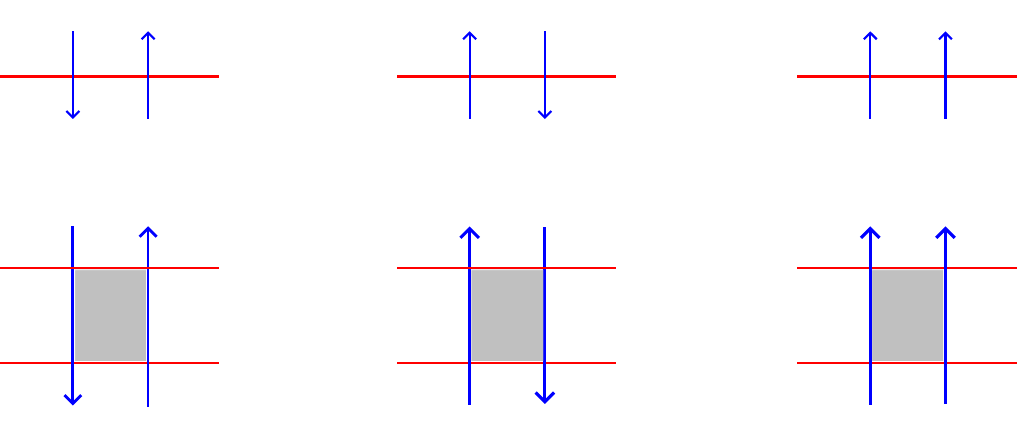}
      \put(22,34){$\color{red}\alpha$}
      \put(16,40){$\color{blue}\beta$}
      \put(4,26){$a.\;$sink arc}
      \put(42,26){$b.\;$source arc}
      \put(80,26){$c.\;$parallel arc}
      \put(2.5,0){$d.\;$sink sector}
      \put(40,0){$e.\;$source sector}
      \put(78.4,0){$f.\;$parallel sector}
  \end{overpic}
  \caption{Arcs and sectors}
  \label{fig:sink_source_parallel}
\end{figure}

When there are no vertical arcs, the boundary of the annulus component consists of 2 $\alpha$-arcs and 2 $\beta$-arcs. By the symmetry from the hyperelliptic involution\footnote{The 2-pointed torus $(\Sigma,z,w)$ can be regarded as a quotient orbifold of the genus 2 closed surface under an involution. The hyperelliptic involution on the genus 2 surface then descends to an involution on $(\Sigma,\alpha,\beta,z,w)$, interchanging the basepoints while reversing the orientations of $\alpha$ and $\beta$. See~\cite{lyu2024knot}, Section~2 for more details.}, there is exactly 1 $\beta$-sink $\alpha$-arc and 1 $\beta$-source $\alpha$-arc. It follows that we can still similarly define the $\beta$-sink tube and $\beta$-source tube. Recall there is exactly one $\beta$-sink bigon and one $\beta$-source bigon by the symmetry.

\begin{defn}[cf. \cite{lyu2024knot}, Definition 2.6]
  In a non-simple, reduced $(1,1)$-diagram without vertical arcs, the ($\beta$-)\textbf{sink tube} is the union of ($\beta$-)sink sectors that connect the ($\beta$-)sink bigon to the annulus. The ($\beta$-)\textbf{source tube} is the union of ($\beta$-)source sectors that connect the ($\beta$-)source bigon to the annulus.
  \label{def:sink_tube}
\end{defn}

The sink tube still contains all the sink sectors. Moreover, since the annulus does not have boundary $\alpha$-arcs that are $\beta$-parallel, there are in fact no $\beta$-parallel sectors or $\beta$-parallel $\alpha$-arcs:

\begin{lem}[cf. \cite{lyu2024knot}, Proposition 2.7]
  The ($\beta$-)sink (resp. source) tube contains all ($\beta$-)sink (resp. source) sectors. Moreover, if the diagram has no vertical arcs, then there are no ($\alpha$- or $\beta$-)parallel arcs or sectors.
  \label{lem:char_sink_tube}
\end{lem}

Recall that the carrying curve of $\beta$ was defined to be isotopic to $\beta$ on the torus $\Sigma$ while passing through only the $\beta$-parallel $\alpha$-arcs (see~\cite{lyu2024knot}, Definition 3.5). Since there are no such $\alpha$-arcs here, we can simply define the carrying curve of $\beta$ to be the core curve $\delta$ of the annulus.

\begin{defn}
    Let $(\Sigma,\alpha,\beta,z,w)$ be a reduced, non-simple $(1,1)$-diagram in $S^1\times S^2$ that contains no vertical arcs when $\alpha$ is placed in the standard horizontal position. Then we define the carrying curve of $\beta$ to be the core curve of the annulus component of $\Sigma-(\alpha\cup\beta)$.
\end{defn}

In particular, the carrying curve of $\alpha$ would be the same curve $\delta$, and replacing $\alpha$ or $\beta$ with $\delta$ would yield a simple diagram (representing the core knot).

\vspace{6pt}

\textbf{Primitive diagrams.} Primitive diagrams serve as one terminal case of our reductions. They are characterized (in the reduction procedure) by the following property: if we replace either curve with its carrying curve, we would get a simple diagram (see~\cite{lyu2024knot}, Proposition 3.9). \textbf{In $S^1\times S^2$, these are exactly the diagrams with no vertical arcs.} We remark that this to some extent also matches the original diagrammatic definition (see~\cite{lyu2024knot}, Proposition 3.3): for the vertical arcs to be of the same direction, there must be no vertical arcs; this in turn forces the rainbow arcs to be alternating since there are no parallel arcs or sectors.

In~\cite{lyu2024knot}, Proposition 3.4 and~\cite{lyu2025persistent}, Lemma 6.5, laminations fully carried by branched surfaces associated to primitive diagrams were constructed directly. This relies on an observation that for primitive diagrams in $S^3$ and lens spaces, the two (innermost) bigons share an endpoint; we then deduce that the sink tube is a spiral, and the branched surface becomes laminar after a single splitting called the ``sink tube push''. In $S^1\times S^2$, things are different: the two bigons could be away from each other. However, the diagram is essentially characterized by a rational tangle, and we could still apply the splitting techniques in the previous papers to construct laminations. We explain this in more detail in Appendix~\ref{app:detail}, after reviewing the necessary background from the earlier papers.

\subsection{Essentiality and incompressibility}
\label{subsec:incomp}

In this subsection, we address the incompressibility issues for $S^1\times S^2$.

\vspace{6pt}

\textbf{Disk leaves and disks of contact.} In order to modify the branched surfaces and laminations we need to use Gabai's gluing lemma (\cite{gabai1992taut}, Operation 2.4.4, see also~\cite{li2002laminar}, Lemma 3.4). This requires the laminations to have no disk leaves. Moreover, to apply Proposition~\ref{prop:blowup} we require the sub-branched surface to have no disks of contact. Since our branched surfaces always have meridional cusps and boundary circles, it suffices to show that our knot complement admits no meridional disks. This is obvious for knots in rational homology spheres. Fortunately, it is still true for non-trivial knots in $S^1\times S^2$:

\begin{lem}
    Let $K\subset S^1\times S^2$ be a $(1,1)$-knot in $S^1\times S^2$ represented by a non-simple reduced $(1,1)$-diagram. Then $S^1\times S^2 - \mathrm{int}(N(K))$ admits no meridional disks.
    \label{lem:meri_disk}
\end{lem}

\begin{proof}
    Suppose $S^1\times S^2 - \mathrm{int}(N(K))$ admits a meridional disk. By sewing $N(K)$ back we get a 2-sphere in $S^1\times S^2$, and it is non-separating since it intersects $K$ only once. Now this 2-sphere can be isotoped to the standard sphere $\{p\}\times S^2$, and $K$ is isotopic to the core knot by the classical 3-dimensional light bulb theorem (see, for example, \cite{rolfsen2003knots}).

    It remains for us to show that if $K\subset S^1\times S^2$ is represented by a non-simple reduced $(1,1)$-diagram $(\Sigma,\alpha,\beta,z,w)$, then it cannot be the core knot. This is probably most easily seen via knot Floer homology\footnote{Since there are homologically essential knots, here we understand the knot Floer homology $\widehat{HFK}(K)$ as the sutured Floer homology $SFH(S^1\times S^2-\mathrm{int}(N(K)), \Gamma(K))$, where $\Gamma(K)$ is a pair of oppositely oriented meridian sutures.}: such a knot $K$ must have non-trivial $\widehat{HFK}$ (whose rank equals $|\alpha\cap\beta|$), while $\widehat{HFK}$ of the core knot vanishes.
\end{proof}

\vspace{6pt}

\textbf{Non-separating spheres.} Laminar branched surfaces were introduced to find \textit{essential} laminations. For the blown-up branched surface $\mathcal{A}$ the essential conditions are hard to make sense of, so we only care whether they fully carry laminations. The tautness of the resulting foliation is guaranteed by certain ``taut features'' of the original branched surface $\mathcal{B}$:

\begin{lem}[see~\cite{lyu2024knot}, Section~4 and~\cite{lyu2025persistent}, Proposition 3.5]
    Let $\mathcal{B}$ be a branched surface constructed as in~\cite{lyu2024knot}, Definition 2.13 or~\cite{lyu2025persistent}, Definitions 3.1$\sim$3.2. Then any lamination $\mathcal{L}$ carried by $\mathcal{B}$ is co-oriented and taut (to every leaf $F$ there is a circle in $N(\mathcal{B})$ transverse to $\mathcal{L}$ while intersecting $F$). In particular, any closed surface carried by $\mathcal{B}$ is non-separating in $N(\mathcal{B})$.
    \label{lem:taut}
\end{lem}

However, we still need to apply the laminar branched surface theory for the intermediate branched surfaces. This is done by the following lemma, where we place our branched surface in a suitable $3$-manifold so that it becomes ``essential'' there:

\begin{lem}[\cite{lyu2024knot}, Lemma 2.11]
  Let $\mathcal{B}$ be a co-oriented branched surface with boundary a union of circles. If $\mathcal{B}$ is sink disk free, not carrying any torus, and no component of the horizontal boundary of $N(\mathcal{B})$ is a disk or a sphere, then $\mathcal{B}$ fully carries a lamination.
  \label{lem:sk_disk_free}
\end{lem}

\begin{proof}
    We can sew in a punctured genus 2 surface to each circular boundary component to get a branched surface $\mathcal{B}'$ without boundary. Consider the regular neighborhood $N(\mathcal{B}')$. All components of $\partial N(\mathcal{B}')$ must have Euler characteristic $\leq 0$. It follows that we can cap these boundary components off by pasting some irreducible 3-manifolds with incompressible boundary. Now $\mathcal{B}'$ is laminar in the resulting closed 3-manifold, thus fully carries a lamination by~\cite{li2002laminar}. $\mathcal{B}$ as a subsurface also fully carries a lamination.
\end{proof}

Notice that one technical condition we imposed here is that the regular neighborhood of our branched surface does not have any sphere horizontal boundary components. For knots in $S^3$ and lens spaces this was easily confirmed using Lemma~\ref{lem:taut}. We show below that the exteriors of non-trivial $(1,1)$-knots in $S^1\times S^2$ do not contain non-separating spheres:

\begin{lem}
    Let $K\subset S^1\times S^2$ be a $(1,1)$-knot in $S^1\times S^2$ represented by a non-simple reduced $(1,1)$-diagram. Then $(S^1\times S^2 - K)$ does not contain a non-separating 2-sphere.
    \label{lem:non-sep_sphere}
\end{lem}

\begin{proof}
    A non-separating sphere implies that the knot is \textit{local}; that is, there exists a knot $K'\subset S^3$ such that $S^1\times S^2-K=(S^3-K')\# (S^1\times S^2)$. Let $g(M)$ denote the Heegaard genus of $M$. Since $K$ is a $(1,1)$-knot, we know that $g(S^3_0(K')\#(S^1\times S^2))\leq 2$. Since the Heegaard genus is additive under connected sums (\cite{haken1968some}), we have $g(S^3_0(K'))\leq 1$. This implies that $K'$ is the unknot (\cite{gabai1987foliationsIII}), and $K$ is the local unknot in $S^1\times S^2$.

    It again remains for us to show that if $K\subset S^1\times S^2$ is represented by a non-simple reduced $(1,1)$-diagram $(\Sigma,\alpha,\beta,z,w)$, then it cannot be the local unknot $U$. We notice that $U$ has simple knot Floer homology, in that $\widehat{HFK}(U)$ has rank 2. On the other hand, there is a unique non-simple reduced $(1,1)$-diagram in $S^1\times S^2$ with $|\alpha\cap\beta|=2$, which is described as $(2,1,0,1)$ using the 4-tuple notation in~\cite{rasmussen2005knot}; in particular the corresponding knot is homologically essential and is not $U$. The lemma then follows.
\end{proof}

\vspace{6pt}

\textbf{Reeb torus.} Another technical condition we imposed in Lemma~\ref{lem:sk_disk_free} is that the branched surface does not carry any torus. This is actually stronger than needed. In fact, to use the laminar branched surface theory, we only need to guarantee that there is no \textit{Reeb torus}, that is, no torus bounding a solid torus in the 3-manifold. Since in the proof of Lemma~\ref{lem:sk_disk_free} the branched surface $\mathcal{B}'$ is obtained by sewing punctured genus 2 surfaces to the boundary of $\mathcal{B}$, any torus carried by $\mathcal{B}'$ is also carried by $\mathcal{B}$. Moreover, the regular neighborhood $N(\mathcal{B})$ embeds in the capped-off closed 3-manifold, so a Reeb torus there is also separating in $N(\mathcal{B})$. This contradicts Lemma~\ref{lem:taut}.

\subsection{Proof of the main theorem}

We can now prove Theorem~\ref{thm:S1S2}.

\begin{proof}[Proof of Theorem~\ref{thm:S1S2}]
    Suppose $K\subset S^1\times S^2$ is a $(1,1)$-knot and $K$ is not the local unknot. If $K$ is represented by a (reduced) coherent diagram, then by Proposition~\ref{prop:coherent} every non-trivial surgery along $K$ yields an L-space. If $K$ is not represented by a coherent diagram, then a reduced $(1,1)$-diagram representing it must be incoherent, and by Proposition~\ref{prop:incoherent} $K$ must be persistently foliar. According to~\cite{ozsvath2004holomorphic,bowden2016approximating,kazez2017c0}, this implies that non-trivial surgeries along $K$ cannot be L-spaces.
\end{proof}

\section{2-bridge links}
\label{sec:2bridge}

In this section, we discuss the L-space conjecture for surgeries on $2$-bridge links, which has recently been studied in detail by~\cite{santoro2026space}. Using $(1,1)$-knot theory, we recover some of the results in~\cite{santoro2026space} and translate the remaining open cases (\cite{santoro2026space}, Problem~9.23) into the $(1,1)$-knot framework. This is based on the following standard observation.

\begin{obs}
    \label{obs:2bridge}
    Performing Dehn surgery on one component of a 2-bridge link turns the other component into a $(1,1)$-knot. In fact, a $(1,1)$-diagram of the knot can be read directly off from the 2-bridge position.

    Identify $S^3$ with $\mathbb{R}^3\cup\{\infty\}$. Let $L$ be a 2-bridge link obtained by gluing 2 (properly embedded) rational tangles along $\{z=0\}\cup\{\infty\}$. Let $\gamma_1,\gamma_2$ be the two tangle strands living in $\{z\geq 0\}$, and $\delta_1,\delta_2$ be the two tangle strands living in $\{z\leq 0\}$, such that $K_1=\gamma_1\cup\delta_1$, $K_2=\gamma_2\cup\delta_2$ are two components of $L$.

    Suppose we perform some $r$-Dehn surgery on $K_1$. Consider the solid torus $T_1=(\{z\geq 0\}- \mathrm{int}( N(\gamma_1)))$, and its boundary torus $\Sigma=\partial T_1$. We claim that $\Sigma$ is a Heegaard torus of the resulting 3-manifold $S^3_{r}(K_1)$, and that $K_2$ is in $(1,1)$ position with respect to $\Sigma$. Indeed, $T_1$ is a solid torus and $K_2\cap T_1=\gamma_2$ is by definition a boundary parallel arc. On the other hand, $S^3_{r}(K_1)-T_1$ is obtained by gluing two solid tori $T_2=(\{z\leq 0\}-\mathrm{int}(N(\delta_1)))$ and $N_{r}(K_1)$ along an annulus $A=\overline{(\partial T_2)|_{\{z<0\}}}$, whose core $l$ is a meridian around $\delta_1$. Notice that $l$ generates the homology of $T_2$. If we think of $l$ as the vertical Seifert fibering, then $T_2$ has no singular fiber, while $N_{r}(K_1)$ may have one singular fiber. It follows that $S^3_{r}(K_1)-T_1=T_2\cup_AN_{r}(K_1)$ is a Seifert manifold with a disk orbifold and at most one singular fiber, and hence is a solid torus. Since $\delta_2$ is already boundary parallel in $T_2$ (where the isotopy to $\partial T_2$ could avoid $A$), it is still boundary parallel in $S^3_{r}(K_1)-T_1$. Hence $K_2$ is indeed in $(1,1)$ position with $\Sigma$.

    \begin{figure}[!hbt]
        \begin{overpic}[scale=0.4]{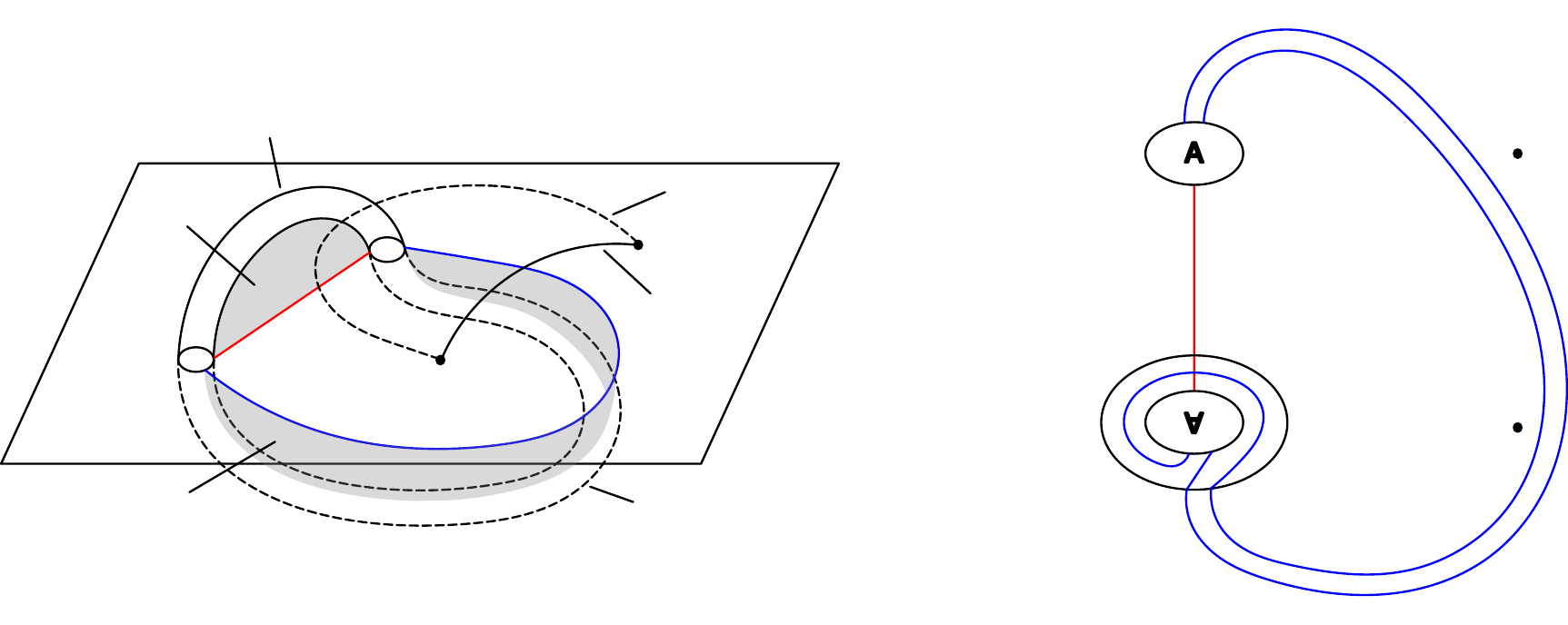}
            \put(9,6.5){$D_2$}
            \put(8,26){$D_1$}
            \put(15,32){$\gamma_1$}
            \put(42.6,27){$\delta_2$}
            \put(42,20){$\gamma_2$}
            \put(40.2,6){$\delta_1$}
            \put(8.8,15.5){\Small $B$}
            \put(25.8,24.5){\Small $A$}
            \put(26.7,15.4){\Small $z$}
            \put(41.2,24){\Small $w$}
            \put(1.4,11.2){\Small $z=0$}
            \put(95,12){\Small $z$}
            \put(97,30.7){\Small $w$}
            \put(68.8,8.8){\Small $B$}
            \put(76.5,21){\color{red} $\partial D_1$}
            \put(97,3.5){\color{blue} $\partial D'$}
        \end{overpic}
        \caption{At left: the Hopf link as the two-bridge link $b(2,1)$; at right: a $(1,1)$-diagram for $K_2$ after doing $\frac{1}{2}$-surgery on $K_1$.}
        \label{fig:2_bridge}
    \end{figure}
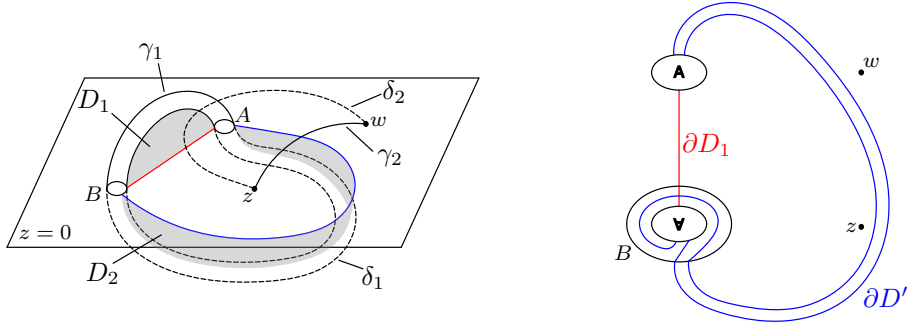

    We now discuss the $(1,1)$-diagram. Suppose the rational tangles are standard, i.e., the projections of the tangle strands onto the $\{z=0\}$ sphere are respectively injective, and their images on the 4-pointed sphere are in tight position (i.e. no trivial bigons). Then the projection of $\gamma_1$ onto $\{z=0\}$ gives a meridional disk $D_1$ of $T_1$ avoiding $\gamma_2$. Similarly, the projection of $\delta_1$ onto $\{z=0\}$ gives a meridional disk $D_2$ of $T_2$ avoiding $\delta_2$. See Figure~\ref{fig:2_bridge} (left). 
    
    To find a meridional disk of $S^3_{r}(K_1)-T_1=T_2\cup_AN_{r}(K_1)$ that avoids $\delta_2$, we first take a meridional disk $D$ of $N_{r}(K_1)$. Its boundary $\partial D$ intersects the annulus $A$ in parallel, coherently oriented arcs. We can arrange these arcs so that they are parallel to $\delta_1$. We can then attach parallel copies of $D_2$ to these arcs to get a meridional disk $D'$ of $S^3_{r}(K_1)-T_1=T_2\cup_AN_{r}(K_1)$. Suppose $K_2\cap\{z=0\}=\{z,w\}$. Then $(\Sigma,\partial D_1, \partial D', z,w)$ is a $(1,1)$-diagram for $K_2\subset S^3_{r}(K_1)$. Denote the two circle ends of the $\gamma_1$-tube as $A,B$. After cutting $\Sigma$ along $A$, we can read the $(1,1)$-diagram in the $\{z=0\}$ plane, see Figure~\ref{fig:2_bridge} (right).
\end{obs}

\begin{rem}
    It is clear that if we are doing the $\frac{m}{n}$-surgery on $K_1$, then there will be $n$ parallel, coherent copies of (projection images of) $\delta_1$ in $\partial D'$ in the $(1,1)$-diagram. $m$ is then further characterized by the twist in the $\gamma_1$-tube. See Figure~\ref{fig:2_bridge}. We note that an untwisted $\gamma_1$-tube does not always indicate the $0$-surgery on $K_1$, although this is the case for the Hopf link. In general, one needs to require the algebraic intersection of $\partial D_1$ and $\partial D'$ to be zero to identify the $0$-surgery (that is, the $(1,1)$-knot in $S^1\times S^2$).
\end{rem}

When working with $(1,1)$-diagrams, we can apply Theorem~\ref{thm:GLV} and Theorem~\ref{thm:S1S2} to find L-space surgeries. For example, we have:

\begin{cor}[cf.~\cite{santoro2026space}, Theorem 1.2]
    Let $L$ be the (unoriented) 2-bridge link $b(p,q)$ obtained by adding $1/0$ bridges to the $q/p$-rational tangle, where $p$ is even, $q$ is odd, $0<|q|\leq \frac{p}{2}$, and $p,q$ are coprime. Then $L$ admits non-trivial L-space surgeries if and only if $p\equiv \pm 1(\mathrm{mod}\;|q|)$.
\end{cor}

\begin{proof}
    If $|q|=\frac{p}{2}$, then $p=2$ and $q=\pm 1$, and the link is the Hopf link with non-trivial L-space surgeries. We henceforth suppose $|q|<\frac{p}{2}$. According to Theorem~\ref{thm:GLV} and Theorem~\ref{thm:S1S2}, $L$ has non-trivial L-space surgeries if and only if some of the $(1,1)$-diagrams in Observation~\ref{obs:2bridge} is coherent. To that end we consider the following model for rational tangles:

    Pick a unit square $abcd$. We can ``blow it up'' to get a bubbled square, which can then be identified with a 4-pointed sphere $(S,a,b,c,d)$ (see the first picture in Figure~\ref{fig:checkerbd}). We fix a frame so that the edges $\overline{ab}$ and $\overline{cd}$ of the bubbled square correspond to the $1/0$ tangle on $S$, while the edges $\overline{ad}$ and $\overline{bc}$ correspond to the $0/1$ tangle. Denote the front square $F$ and the back square $\bar{F}$. Notice that $S=F\cup \bar{F}$.

    There is a branched covering $\mathbb{R}^2\rightarrow S$, such that the preimages of $F$ and $\bar{F}$ are unit squares that form a checkerboard, where the preimage of $\overline{ab}$ is $\{x=2k,k\in\mathbb{Z}\}$, and the preimage of $\overline{cd}$ is $\{x=2k+1,k\in\mathbb{Z}\}$. See the top 2 pictures of Figure~\ref{fig:checkerbd}. A standard strand $\beta$ of the $q/p$ tangle with endpoints $a,b$ on $S$ can then be lifted to $\mathbb{R}^2$ as a segment $l_{pq}$ with endpoints $(0,0)$ and $(p,q)$. See the bottom 2 pictures of Figure~\ref{fig:checkerbd} for an example, where $p=18$ and $q=5$.

    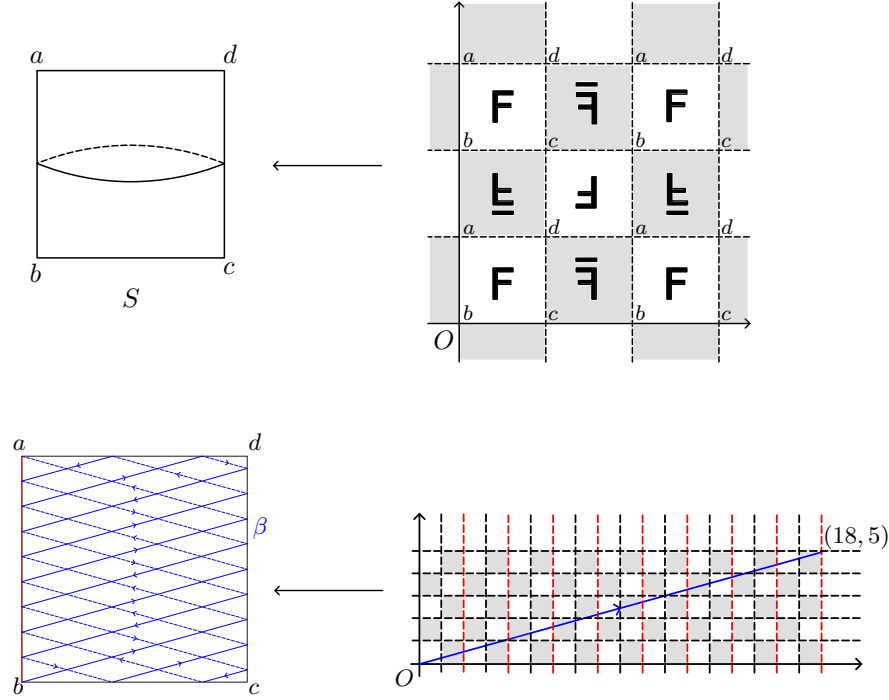
\begin{figure}[!hbt]
        \begin{overpic}[scale=0.5]{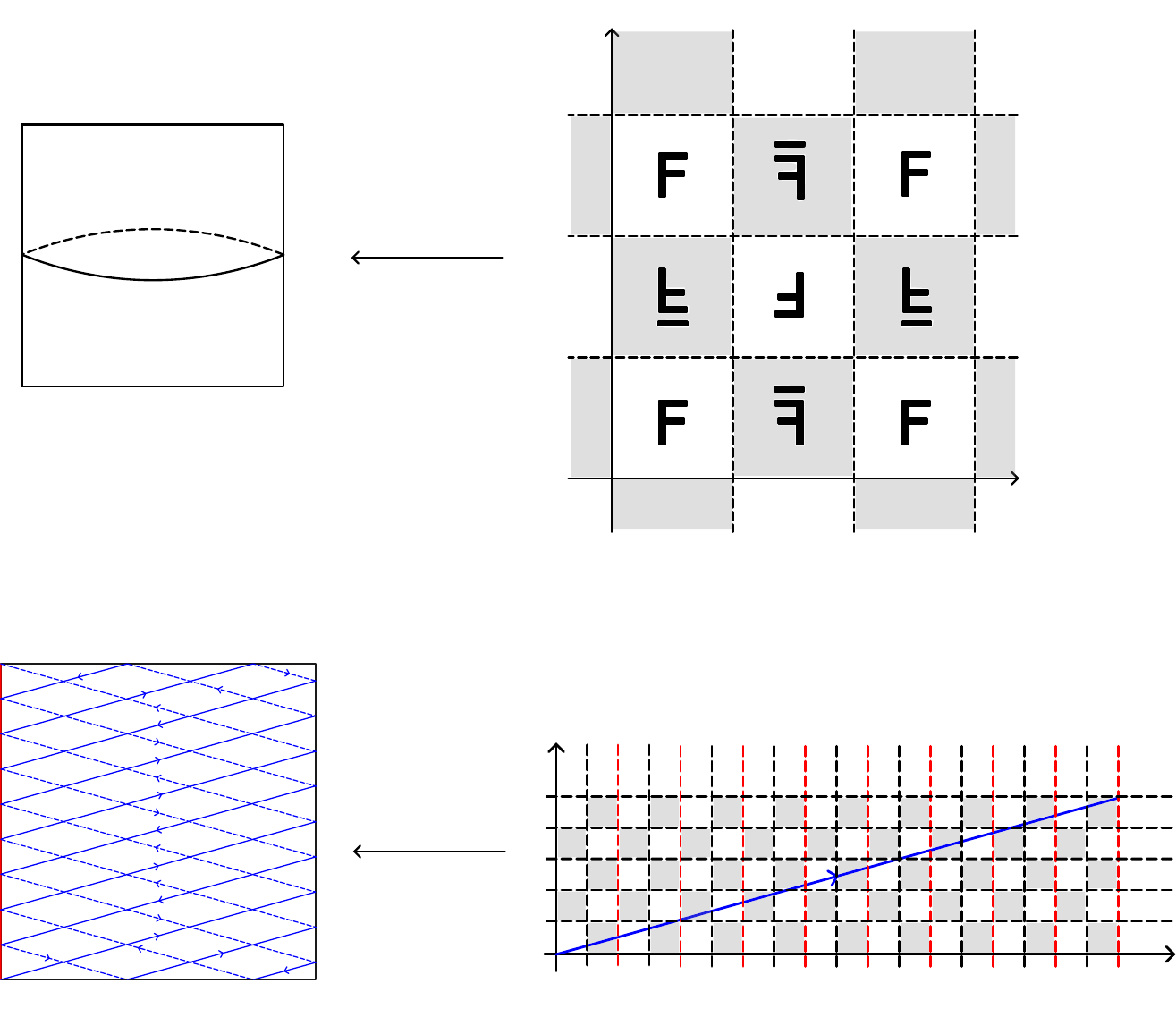}
            \put(1,76){$a$}
            \put(1,50){$b$}
            \put(24,51){$c$}
            \put(24,76){$d$}
            \put(12,47){$S$}
            \put(49,42){$O$}
            \put(52.5,55.7){\Small $a$}
            \put(52.5,45.5){\Small $b$}
            \put(62.6,45.5){\Small $c$}
            \put(62.6,55.7){\Small $d$}
            \put(52.5,76.2){\Small $a$}
            \put(52.5,66){\Small $b$}
            \put(62.6,66){\Small $c$}
            \put(62.6,76.2){\Small $d$}
            \put(73,55.7){\Small $a$}
            \put(73,45.5){\Small $b$}
            \put(83.2,45.5){\Small $c$}
            \put(83.2,55.7){\Small $d$}
            \put(73,76.2){\Small $a$}
            \put(73,66){\Small $b$}
            \put(83.2,66){\Small $c$}
            \put(83.2,76.2){\Small $d$}
            \put(-1,30){\small $a$}
            \put(-1,1){\small $b$}
            \put(27,1){\small $c$}
            \put(27,30){\small $d$}
            \put(27.5,20){\small \color{blue} $\beta$}
            \put(44.4,2){\small $O$}
            \put(95.3,19){\small $(18,5)$}
        \end{overpic}
        \caption{The checkerboard covering}
        \label{fig:checkerbd}
    \end{figure}

    The intersection points of $\beta$ and $\overline{ab}$ on $S$ are lifted to $\mathbb{R}^2$ as the intersection points of $l_{pq}$ and $\{x=2k,k\in\mathbb{Z}\}$. The sign of the intersection on $S$ is recorded on $\mathbb{R}^2$ as follows: for an intersection point $(x,y)\in \mathbb{R}^2$ in the interior of $l_{pq}$, the corresponding intersection back on $S$ is positive if and only if $\lfloor |y|\rfloor$ is even. The same rule applies to intersection points of $\beta$ and $\overline{cd}$.

    Notice that our standard strand $\beta$ is in tight position with $\overline{ab}$ on $S$. Now $\overline{ab}$ cuts $\beta$ into several subarcs. A \textit{rainbow} arc is one whose endpoints are in the interior of $\beta$, such that the signs of intersection at the two endpoints differ. Moreover, the sign of the rainbow arc can be understood as the sign of intersection of the rainbow arc and $\overline{cd}$ (they must intersect, for otherwise the rainbow arc would bound a trivial bigon with $\overline{ab}$). To guarantee an incoherent diagram, it suffices to find rainbow arcs of different signs.

    \vspace{6pt}
    
    \textbf{Claim.} There exist rainbow arcs of different signs if and only if $p\not\equiv\pm1(\mathrm{mod}\;|q|)$.

    \begin{proof}[Proof of the claim]
        We can lift this story to $\mathbb{R}^2$. The subarcs are lifted to the sub-segments of $l_{pq}$ cut out by $\{x=2k,k\in\mathbb{Z}\}$. For $1\leq i<p$, let $P_i=(i,\frac{iq}{p})$ be the intersection point of $l_{pq}$ and $\{x=i\}$. Moreover, for $P=(x,y)\in \mathbb{R}^2$, denote $f(P):=\lfloor |y|\rfloor(\mathrm{mod}\;2)$. A rainbow arc is then some segment $P_{2j}P_{2j+2}$ ($1\leq j\le \frac{p}{2}-2$) where $f(P_{2j})\neq f(P_{2j+2})$, and the sign is indicated by $f(P_{2j+1})$.

        Recall that we set $|q|<\frac{p}{2}$. Then we have $\lfloor \frac{(i+1)|q|}{p}\rfloor-\lfloor \frac{i|q|}{p}\rfloor\in \{0,1\}$ for $1\leq i<p$, and moreover $\lfloor \frac{(2j+2)|q|}{p}\rfloor-\lfloor \frac{2j|q|}{p}\rfloor\in \{0,1\}$ for $1\leq j\le \frac{p}{2}-2$. Now in this case the rainbow arcs are in 1-1 correspondence to the $j$'s such that $\lfloor \frac{(2j+2)|q|}{p}\rfloor-\lfloor \frac{2j|q|}{p}\rfloor=1$, and hence are in 1-1 correspondence to the $i$'s with $2\le i< p-2$ such that $\lfloor \frac{(i+1)|q|}{p}\rfloor-\lfloor \frac{i|q|}{p}\rfloor=1$. The sign and side of the rainbow arc can be read as follows: if $\lfloor \frac{i|q|}{p}\rfloor$ is even, then it corresponds to a rainbow arc on the positive side of $\overline{ab}$, where the rainbow arc is positive if and only if $i$ is odd; if $\lfloor \frac{i|q|}{p}\rfloor$ is odd, then it corresponds to a rainbow arc on the negative side of $\overline{ab}$, where the rainbow arc is positive if and only if $i$ is even. In particular the rainbow arc is positive if and only if ($i+\lfloor \frac{i|q|}{p}\rfloor$) is odd.

        Since $p,q$ are coprime, $\frac{i|q|}{p}\notin \mathbb{Z}$ for $1\leq i<p$. Suppose $\frac{i|q|}{p}<t< \frac{(i+1)|q|}{p}$ for some $t\in \mathbb{Z}$. Then $t=\lfloor \frac{i|q|}{p}\rfloor+1$ and $i=\lfloor \frac{pt}{|q|}\rfloor$. Denote $S_t:=t+\lfloor \frac{pt}{|q|}\rfloor-1=i+\lfloor \frac{i|q|}{p}\rfloor$. As $i$ ranges over $2\leq i<p-2$, the corresponding integer $t$ ranges over $1\le t\le |q|-1$. In particular, if $|q|=1$ then there are no such $t$, and in fact there are no rainbow arcs. From now on suppose $|q|\geq 3$.
        
        Suppose we have $p=m|q|+r$ for $m\in\mathbb{Z}^+$ and $0<r<|q|$. If $m$ is odd (then $r$ is odd), then $S_t\equiv \lfloor \frac{rt}{|q|}\rfloor-1(\mathrm{mod}\;2)$, and $S_t(\mathrm{mod}\;2)$ takes a single value if and only if $r=1$. If $m$ is even (then $r$ is even), then $S_t\equiv \lfloor \frac{rt}{|q|}\rfloor+t-1\equiv \lfloor \frac{(r-|q|)t}{|q|}\rfloor-1(\mathrm{mod}\;2)$, which takes a single value if and only if $r-|q|=-1$. It follows that there are rainbow arcs of different signs if and only if $p\not\equiv\pm1(\mathrm{mod}\;|q|)$.
    \end{proof}

    It follows from the claim that if $p\not\equiv\pm1(\mathrm{mod}\;|q|)$, then the $(1,1)$-diagrams in Observation~\ref{obs:2bridge} are all incoherent, regardless of the twist of the $\gamma_1$-tube; hence the corresponding 2-bridge link does not admit non-trivial L-space surgeries. On the other hand, if $p\equiv\pm1(\mathrm{mod}\;|q|)$, we can pick the integer surgery corresponding to the untwisted $\gamma_1$-tube and consider the $(1,1)$-diagram. To show that the diagram is coherent, it suffices to check that the $\partial D'$-arc passing through the $\gamma_1$-tube is not a rainbow arc (this is the only $\partial D'$-arc on the $(1,1)$-diagram that is not a subarc of $\beta$ we discussed before). On $\mathbb{R}^2$ we can regard this arc as obtained by joining the two sub-segments $P_0P_2$ and $P_{p-2}P_p$ at $P_0$ and $P_p$, with two endpoints $P_2$ and $P_{p-2}$. Since $q$ is odd, $f(P_2)=0$ and $f(P_{p-2})=|q|-1$ are both even. Hence back on $S$ the two intersections are of the same sign, and the $\partial D'$-arc is indeed not a rainbow arc. It follows that the $(1,1)$-diagram is coherent, and the corresponding knot admits non-trivial L-space surgeries. Hence the 2-bridge link also admits non-trivial L-space surgeries. This completes the proof.
\end{proof}

Once we have determined the L-space surgeries, the existence of taut foliations immediately follows from~\cite{lyu2025persistent} and Proposition~\ref{prop:incoherent}:

\begin{cor}[cf.~\cite{santoro2026space}, Theorem 1.2]
    For 2-bridge links without non-trivial L-space surgeries, in fact every non-trivial surgery admits a co-oriented taut foliation.
\end{cor}

For 2-bridge links with L-space surgeries, it is possible to determine the 2-bridge L-space links (these correspond to the case $p\equiv -1(\mathrm{mod}\;|q|)$, where only \textit{one direction} of the twisted $\gamma_1$-tubes gives incoherent diagrams), but determining the L-space surgery set of the link boils down to verifying whether certain surgeries give an L-space or calculating the Turaev torsions for the $(1,1)$-knot complements, as Santoro-Zhou did in~\cite{santoro2026space}. The existence problem of co-oriented taut foliations would be completely resolved by the following conjecture for $(1,1)$-knots:

\begin{conj}[cf. the L-space conjecture]
    A Dehn surgery on a $(1,1)$ L-space knot admits a co-oriented taut foliation if and only if it is not an L-space.
    \label{conj:11Lsp}
\end{conj}

It is worth noting that Conjecture~\ref{conj:11Lsp} is verified for $(1,1)$ L-space knots in $S^3$ (\cite{nie2021explicit,nie2026positive}) and $S^1\times S^2$ (\cite{rasmussen2017floer}). However, it remains open for $(1,1)$ L-space knots in lens spaces.

\appendix
\section{Primitive diagrams, branched surfaces, and laminations}
\label{app:detail}

\subsection{Background}
\label{appsec:backgr}

In this subsection, we review the necessary background from~\cite{lyu2024knot,lyu2025persistent} for later use.

\vspace{6pt}

\textbf{Heegaard branched surfaces.} We first review the definitions of our ``Heegaard'' branched surfaces used to construct taut foliations. The name comes from the fact that they are constructed from Heegaard diagrams.

Let $(\Sigma,\alpha,\beta,z,w)$ be a reduced $(1,1)$-diagram defining a $(1,1)$-knot $K$ (in $S^1\times S^2$). Pick the Heegaard torus $\Sigma$ and attach 2 disks to different sides of $\Sigma$ along $\alpha$ and $\beta$ separately. Then remove two small open disks at the 2 basepoints $z,w$. Call the resulting 2-complex $B$. It is worth noticing that $B$ embeds in the knot complement, and in fact the regular neighborhood of $B$ is homeomorphic to the knot complement.

We still need to smooth $B$ along $\alpha$ and $\beta$ to get a branched surface. Fixing an orientation of $\Sigma$, we can further orient the curves $\alpha,\beta$. For a choice of the orientations of $(\alpha,\beta)$, we can smooth $B$ accordingly, such that the branch direction always points to the \textit{left} of the curves on $\Sigma$. By the symmetry from the hyperelliptic involution, there are essentially 2 different smoothings.

\begin{defn}[cf.~\cite{lyu2024knot}, Definition 2.13]
    Suppose $(\Sigma,\alpha,\beta,z,w)$ is non-simple. Pick orientations of $\alpha,\beta$ such that on $\Sigma$ the (innermost) bigon containing $w$ has clockwise boundary. The corresponding smoothing of $B$ would give a branched surface where this (punctured) bigon becomes a \textit{source sector}. We can then remove this source sector. The resulting branched surface $\mathcal{B}$ is called the \textit{(Heegaard) branched surface associated to $(\Sigma,\alpha,\beta,z,w)$}.
    \label{def:asso_br}
\end{defn}

\begin{defn}[cf.~\cite{lyu2025persistent}, Definition 3.1]
    Suppose $(\Sigma,\alpha,\beta,z,w)$ is incoherent. Pick orientations of $\alpha,\beta$ so that on $\Sigma$ the innermost bigon containing $w$ is $\alpha$-source while $\beta$-sink (recall the definitions in~\ref{def:sink_source_parallel}). The incoherence of the diagram then guarantees that there exist some quadrilateral sector(s) with clockwise boundary, indicating that they are source sectors in the branched surface after smoothing. Let $T$ be any such quadrilateral source sector. We can then reverse the co-orientation of $T$ to get a new branched surface $\mathcal{B}_T$. This branched surface $\mathcal{B}_T$ is called the \textit{type I modified Heegaard branched surface associated to $T$}.
    \label{def:modified_br}
\end{defn}

\vspace{6pt}

\textbf{Sink tube push.} We now review the core splitting technique in~\cite{lyu2024knot,lyu2025persistent}. We will try to summarize it in a more general setting with some abuse of notation (which indicates the prototypes of these definitions).

Let $\mathcal{B}$ be a branched surface and $\Sigma\subset N(\mathcal{B})$ be a compact \textit{subsurface}\footnote{Here one should think of $\mathcal{B}$ as the branched surface obtained by directly smoothing the $2$-complex $B$, which still carries the (punctured) Heegaard torus $\Sigma$.}. The branch locus on $\Sigma$ divides $\Sigma$ into several branch sectors. For our purpose suppose the branch locus on $\Sigma$ consists of $2$ immersed curves $\alpha,\beta$, and we orient them so that branch directions always point to the left. Then we can define sink, source, and parallel arcs and sectors as in Definition~\ref{def:sink_source_parallel}. We remark that with our conventions a quadrilateral sector is a sink disk of the branched surface if and only if it is both $\alpha$-sink and $\beta$-sink as in Definition~\ref{def:sink_source_parallel}.

To define a sink tube push, we need a well-defined sink tube. Recall that a $\beta$-sink tube is defined when among the boundary $\alpha$-arcs of the non-quadrilateral sectors, there are exactly $2$ that are $\beta$-sink. Moreover, such a $\beta$-sink tube necessarily contains all the $\beta$-sink (quadrilateral) sectors. In particular, it contains all the quadrilateral sink disk candidates on $\Sigma$. The idea of the sink tube push is to destroy all such candidates with a single splitting.

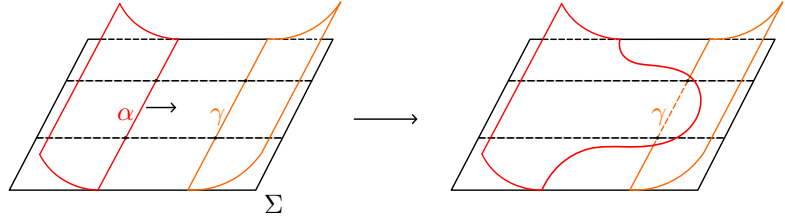
\begin{figure}[!hbt]
  \begin{overpic}[scale=0.7]{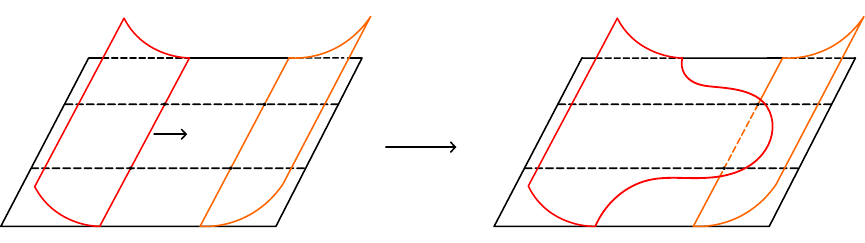}
      \put(14,11){$\color{red} \alpha$}
      \put(26,11){$\color{orange} \gamma$}
      \put(83,11){$\color{orange} \gamma$}
      \put(33,-1){$\Sigma$}
  \end{overpic}
  \caption{Pushing the $\alpha$-arc onto the $\gamma$-arc}
  \label{fig:pushing_arcs}
\end{figure}

We can assume that our splittings do not destroy $\Sigma$, that is, the branched surface after splitting still carries $\Sigma$. The advantage is that we can then describe the 3-dimensional splittings as some 2-dimensional operations on $\Sigma$ which we call ``pushing arcs''. See Figure~\ref{fig:pushing_arcs}. When observed on $\Sigma$, it looks like the $\alpha$-arc is pushed over the $\gamma$-arc onto the sector bounded by $\gamma$ (thus no longer lives on $\Sigma$), and we say that the $\alpha$-arc is \textit{pushed onto} the $\gamma$-arc. A $\beta$-sink tube push is then pushing one long boundary $\beta$-arc of the sink tube onto the other.

Suppose $\mathcal{B}'$ is some branched surface modified from $\mathcal{B}$ (by removing or reversing a source sector on $\Sigma$). As long as we are not moving the boundary of the modified source sector, we can still talk about ``pushing arcs'' on $\Sigma$. When exactly one long boundary $\beta$-arc of the $\beta$-sink tube intersects the modified source sector, the sink tube push is necessarily defined by pushing the other long $\beta$-arc onto this one. This is the ``sink tube push'' defined in~\cite{lyu2024knot}, Subsection 3.4 or~\cite{lyu2025persistent}, Subsection 6.3. 

\subsection{Primitive diagrams in $S^1\times S^2$}
\label{appsec:primitive}

In this subsection we prove that the branched surfaces we constructed for primitive $(1,1)$-diagrams in $S^1\times S^2$ fully carry laminations.

\begin{prop}
    Let $(\Sigma,\alpha,\beta,z,w)$ be a reduced, non-simple, and primitive diagram in $S^1\times S^2$. Then the branched surface associated to it (as in Definition~\ref{def:asso_br}) fully carries a lamination.
    \label{prop:asso_primi}
\end{prop}

\begin{proof}
    Recall our conventions that $w$ is contained in the (removed) source bigon. Suppose we put $\alpha$ in the standard horizontal position. Recall Subsection~\ref{subsec:diadecomp} that $(\Sigma,\alpha,\beta,z,w)$ has no vertical arcs, and there is an annulus sector in $\Sigma-(\alpha\cup\beta)$. Let $\delta$ be the core curve of the annulus sector. We can formally cut $\Sigma$ along $\delta$\footnote{This is just to simplify the models on which we define our splittings; we do not actually cut the branched surface along $\delta$.} and get an annulus with two boundary circles $\delta_z,\delta_w$ (where we suppose $z$ is contained in the annulus bounded by $\delta_z$ and $\alpha$). We can then collapse $\delta_z$ and $\delta_w$ into two points and get a 4-pointed sphere $(S,z,w,\delta_z,\delta_w)$, see the first two pictures of Figure~\ref{fig:asso_primi_red}. On the sphere $S$, $\alpha$ becomes an essential curve separating $z,\delta_z$ and $w,\delta_w$, and $\beta$ also becomes an essential curve separating $z$ and $w$. We can still describe certain splittings of the branched surface as pushing the $\beta$-arcs on $(S,z,w,\delta_z,\delta_w)$, as long as we do not move the boundary of the bigon containing $w$. In fact, by reducing from the torus $\Sigma$ to the sphere $S$ we actually further require that we do not push arcs across $\delta$.

    The reason for reducing $\Sigma$ to $S$ is that curves on the 4-pointed sphere $S$ are better characterized: each essential curve separates two strands of a rational tangle, and is characterized by a rational number (with a fixed frame). Here we fix a frame so that the $\alpha$ curve corresponds to $\frac{1}{0}$. Since the original $(1,1)$-diagram is reduced, $\alpha,\beta$ are still in tight position on the 4-pointed sphere $S$, and $\beta$ is then \textit{characterized by} a rational number $\frac{p}{q}$. By carefully choosing the frame, we may assume that $0\leq \frac{p}{q}\leq 1$. Again see the top right picture of Figure~\ref{fig:asso_primi_red}, where we pick the horizontal edges of the square to represent $\frac{0}{1}$, and $\beta$ corresponds to $\frac{3}{5}$ in this frame.

    \begin{figure}[!hbt]
        \begin{overpic}[scale=0.5]{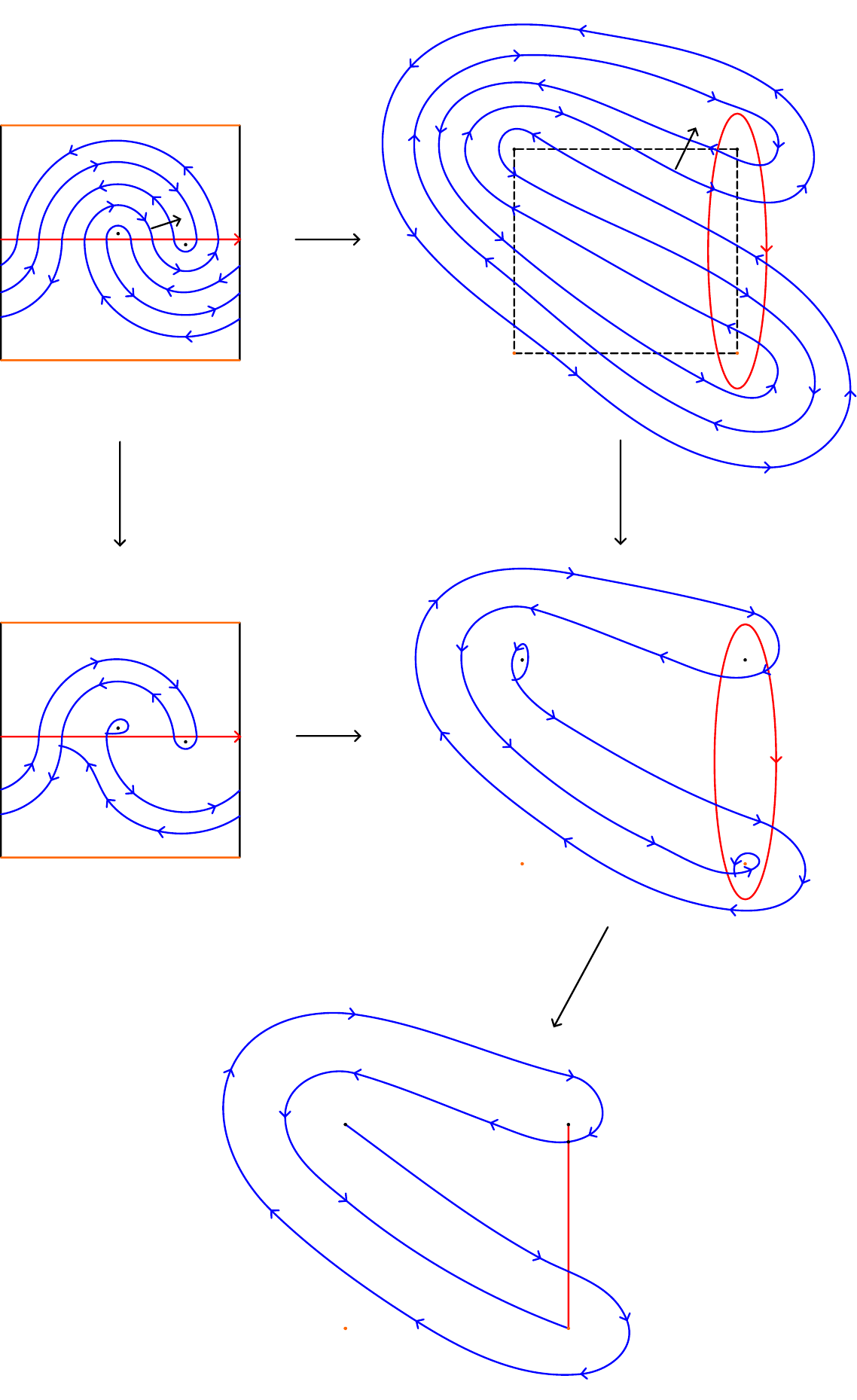}
            \put(8,82.3){\Tiny $z$}
            \put(12.6,83){\Tiny $w$}
            \put(10,73.2){\Tiny \color{orange}$\delta_w$}
            \put(10,91.5){\Tiny \color{orange}$\delta_z$}
            \put(17.2,83){\Tiny \color{red}$\alpha$}
            \put(7.6,76.7){\Tiny \color{blue}$\beta$}
            \put(19,81.7){\Tiny collapsing $\delta$'s}
            \put(37,88.6){\Tiny $z$}
            \put(36.5,73.3){\Tiny \color{orange}$\delta_z$}
            \put(52,89.8){\Tiny $w$}
            \put(51.8,73.5){\Tiny \color{orange}$\delta_w$}
            \put(9.2,64.5){\Tiny sink tube push}
            \put(45,64.5){\Tiny sink tube push}
            \put(7,48){\Tiny $z$}
            \put(12.6,47.5){\Tiny $w$}
            \put(10,37.7){\Tiny \color{orange}$\delta_w$}
            \put(10,56){\Tiny \color{orange}$\delta_z$}
            \put(17.2,47.5){\Tiny \color{red}$\alpha$}
            \put(7.6,41.2){\Tiny \color{blue}$\beta$}
            \put(19,46.2){\Tiny collapsing $\delta$'s}
            \put(38,52){\Tiny $z$}
            \put(37.3,37.2){\Tiny \color{orange}$\delta_z$}
            \put(53,53.3){\Tiny $w$}
            \put(52.4,36.7){\Tiny \color{orange}$\delta_w$}
            \put(42,30){\Tiny collapsing small circles and $\alpha$}
            \put(25,19.8){\Tiny $z$}
            \put(25,5.3){\Tiny \color{orange}$\delta_z$}
            \put(40.5,20.1){\Tiny $w$}
            \put(40.8,4.7){\Tiny \color{orange}$\delta_w$}
            \put(40.8,17.5){\Tiny $e$}
        \end{overpic}
        \caption{Rational tangle strand from knot (10,5,0,3)}
        \label{fig:asso_primi_red}
    \end{figure}
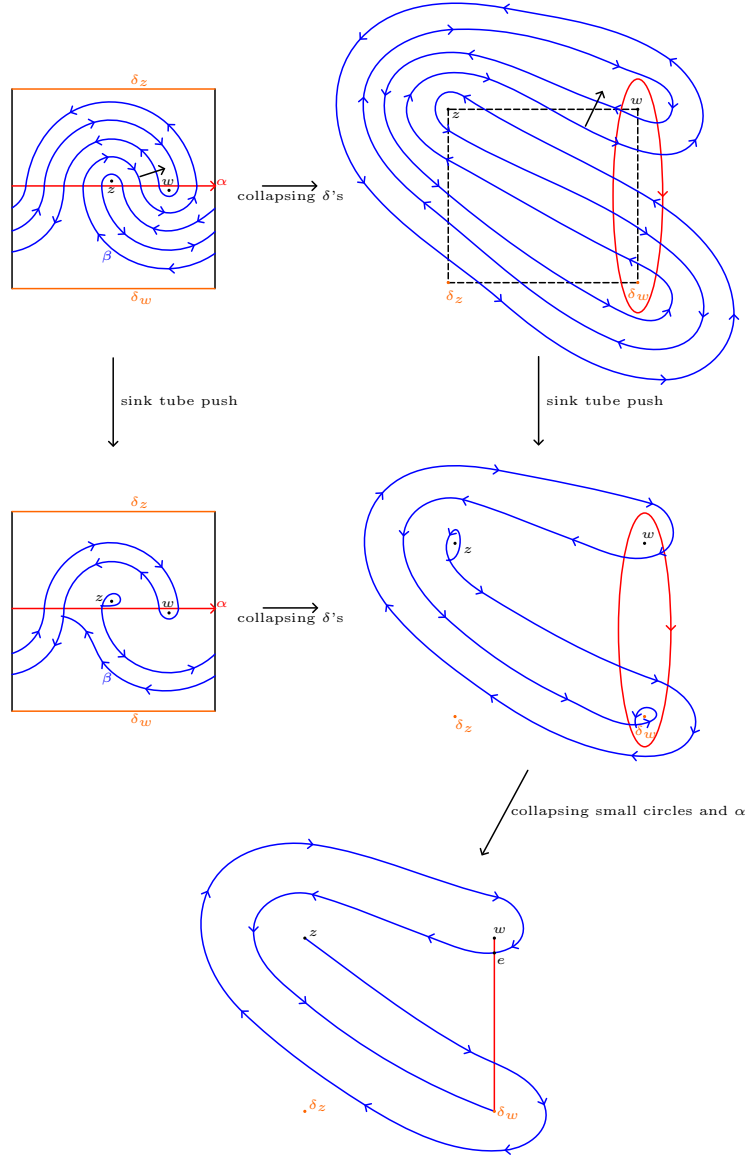

    The ($\beta$-)sink tube push on $\Sigma$ also descends to $S$, see the first 4 pictures of Figure~\ref{fig:asso_primi_red}, where the black arrows in the first 2 pictures indicate the direction of the splitting. Recall that $z$ is in the sink direction of $\beta$. After the sink tube push, the remaining $\beta$-curves on $\Sigma$ consist of a small circle containing $z$, a circle parallel to $\delta$, and an arc connecting the two circles. It follows that on $S$ the remaining $\beta$-curves consist of 2 small circles, each bounding a basepoint in the sink direction, and an arc connecting the two circles, characterized as a strand of the rational tangle $\frac{p}{q}$, see the middle right picture of Figure~\ref{fig:asso_primi_red}.

    Since the two small circles each bound a small disk with a basepoint inside in the sink direction, we can collapse them to the basepoints inside respectively. For simplicity, we can also collapse the $\alpha$-curve on $S$ in its source direction to the arc $w\delta_w$. The boundary $\beta$-arc of the (removed) source bigon is collapsed to a single intersection point $e$ next to $w$. See the last picture of Figure~\ref{fig:asso_primi_red}. Notice that no sink disk candidate is collapsed when collapsing $\alpha$ in its source direction.

    The situation is now essentially the same as that in~\cite{lyu2024knot}, Subsection~3.5, where $\alpha,\beta$ were reduced to two rational tangle strands, and we defined recursive sink tube pushes fixing the intersection point $e$ to eliminate the sink disks. The same recursive argument applies here.

    After the recursive sink tube pushes on $S$ we may assume that no remaining sector on $S$ represents a sink disk of the branched surface. To check that the resulting branched surface is sink-disk-free, we need only consider the branch sectors corresponding to the collapsed pointed disks during the recursive sink tube pushes on $S$ (e.g. see the two collapsed disks with basepoints $z$ and $\delta_w$ respectively in the middle right picture of Figure~\ref{fig:asso_primi_red}). Here the old argument in~\cite{lyu2024knot} still works, except for the slight difference that $\delta$ does not refer to any cusp or boundary. However, since our splittings never cross $\delta$, the branch sector containing $\delta$ remains an annulus sector after the splittings, and is therefore not a sink disk.

    It follows that after the splittings we get a branched surface with no sink disks. In light of Subsection~\ref{subsec:incomp}, we may safely apply Lemma~\ref{lem:sk_disk_free} to conclude that our branched surface fully carries a lamination. So does the original associated branched surface before the splittings.
\end{proof}

\begin{prop}
    Let $(\Sigma,\alpha,\beta,z,w)$ be a reduced, incoherent, and primitive diagram in $S^1\times S^2$. Let $S_0$ be the quadrilateral sector that shares a boundary $\beta$-arc with the $\beta$-sink bigon containing $w$. Then $S_0$ is a source sector, and the type I modified branched surface $\mathcal{B}_{S_0}$ associated to $S_0$ (see Definition~\ref{def:modified_br}) fully carries a lamination.
    \label{prop:modi_primi}
\end{prop}

\begin{proof}

    \begin{figure}[!hbt]
        \begin{overpic}[scale=0.5]{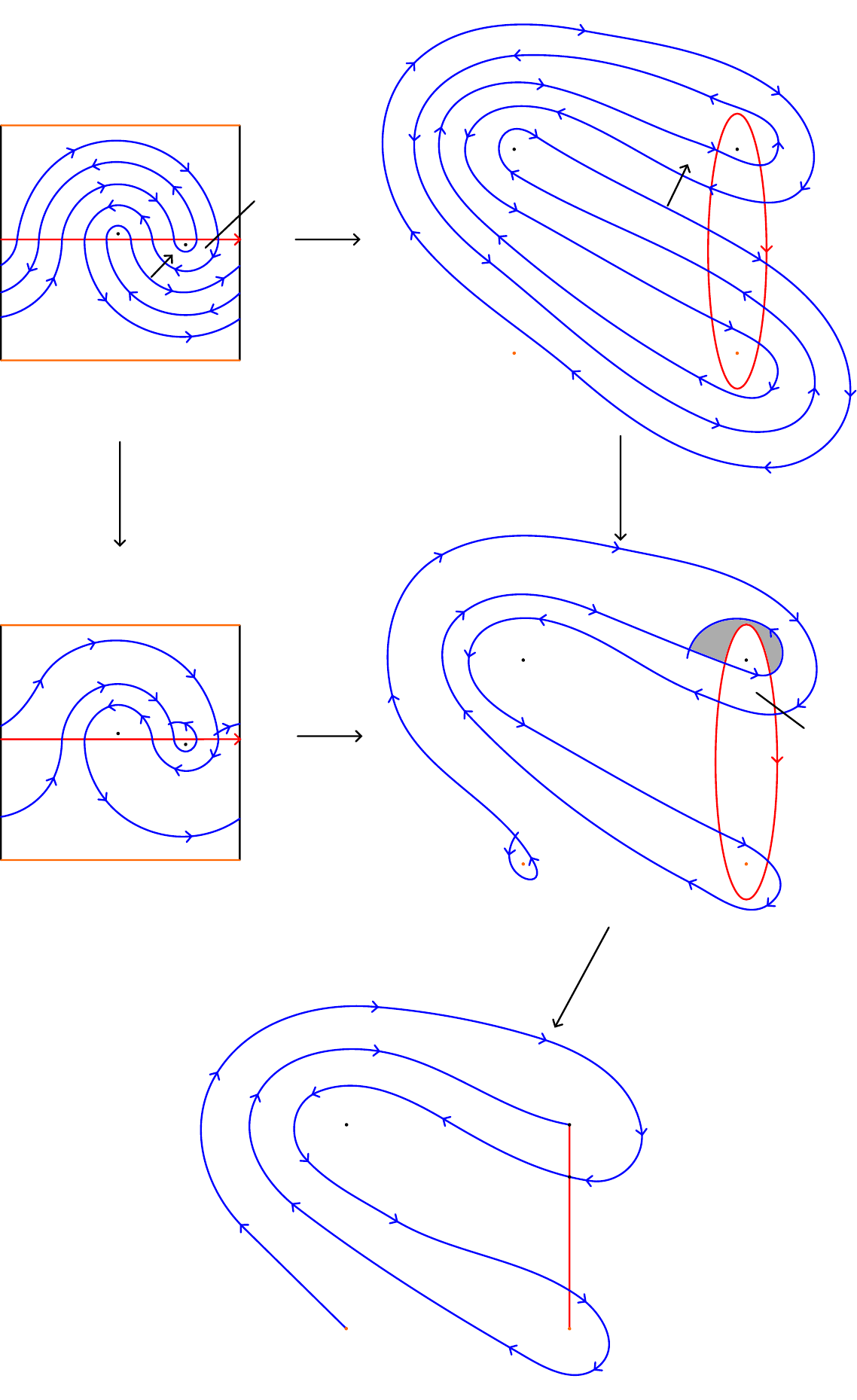}
            \put(8,82.3){\Tiny $z$}
            \put(12.6,83){\Tiny $w$}
            \put(10,73.2){\Tiny \color{orange}$\delta_w$}
            \put(10,91.5){\Tiny \color{orange}$\delta_z$}
            \put(17.2,83){\Tiny \color{red}$\alpha$}
            \put(7.6,76.7){\Tiny \color{blue}$\beta$}
            \put(18,86){\Tiny $S_0$}
            \put(19,81.7){\Tiny collapsing $\delta$'s}
            \put(37,88.6){\Tiny $z$}
            \put(36.5,73.3){\Tiny \color{orange}$\delta_z$}
            \put(52,89.8){\Tiny $w$}
            \put(51.8,73.5){\Tiny \color{orange}$\delta_w$}
            \put(9.2,65.5){\Tiny (modified)}
            \put(9.2,64){\Tiny sink tube push}
            \put(45,65.5){\Tiny (modified)}
            \put(45,64){\Tiny sink tube push}
            \put(7.8,48){\Tiny $z$}
            \put(12.6,47.3){\Tiny $w$}
            \put(10,37.4){\Tiny \color{orange}$\delta_w$}
            \put(10,56){\Tiny \color{orange}$\delta_z$}
            \put(17.2,46.5){\Tiny \color{red}$\alpha$}
            \put(7.6,41.2){\Tiny \color{blue}$\beta$}
            \put(19,46.2){\Tiny collapsing $\delta$'s}
            \put(38,52){\Tiny $z$}
            \put(35.5,37.2){\Tiny \color{orange}$\delta_z$}
            \put(53,53.3){\Tiny $w$}
            \put(52.4,36.9){\Tiny \color{orange}$\delta_w$}
            \put(57.7,47){\Tiny $S_0$}
            \put(42,30){\Tiny collapsing small circles and $\alpha$}
            \put(25,19.8){\Tiny $z$}
            \put(25,5.3){\Tiny \color{orange}$\delta_z$}
            \put(40.5,20.1){\Tiny $w$}
            \put(40.8,4.7){\Tiny \color{orange}$\delta_w$}
            \put(40.8,16){\Tiny $e$}
        \end{overpic}
        \caption{Rational tangle strand from knot (10,5,0,3) again}
        \label{fig:modified_primi_red}
    \end{figure}
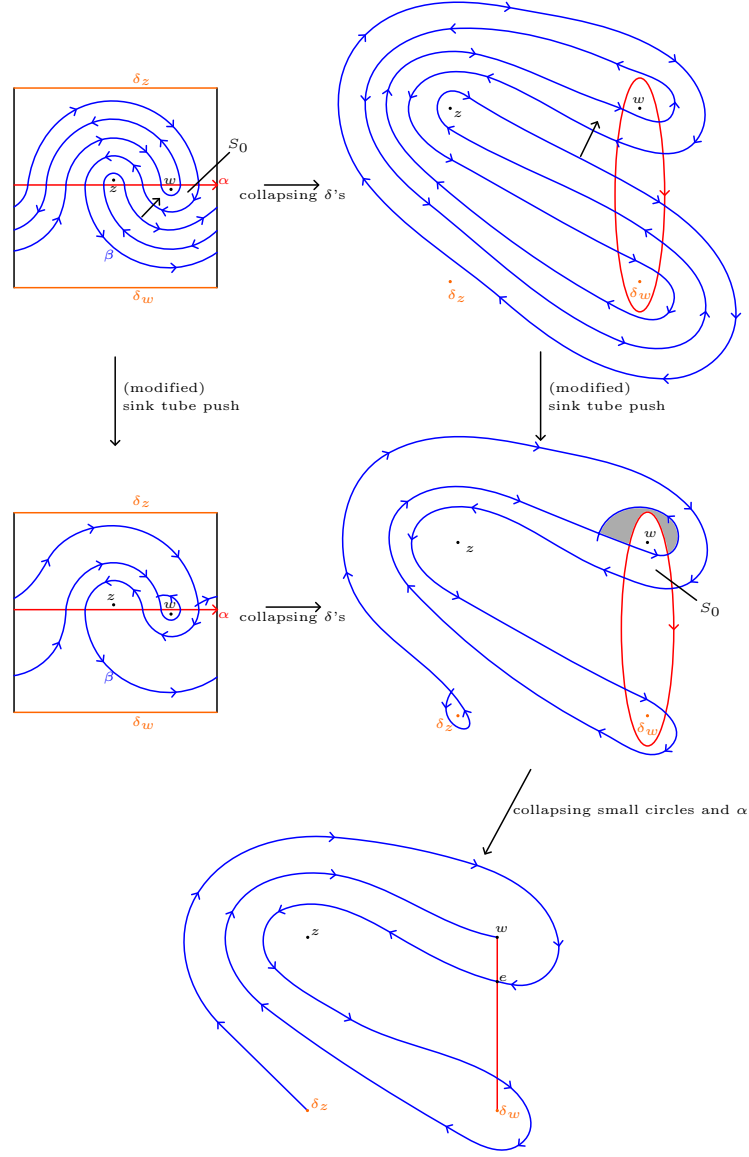

    Suppose we put $\alpha$ in the standard horizontal position. Since the diagram is incoherent, there must be multiple rainbow arcs on each side of $\alpha$ (hence $S_0$ must exist). Recall our conventions that the bigon containing $w$ is $\alpha$-source and $\beta$-sink. Since the rainbow arcs are alternating, $S_0$ must be both $\alpha$-source and $\beta$-source, thus a source sector of the branched surface obtained by smoothing the 2-complex $B$. It follows that the modified branched surface $\mathcal{B}_{S_0}$ is well defined.

    One additional step is needed here. After reversing $S_0$, the bigon containing $w$ becomes a \textit{source sector} in $\mathcal{B}_{S_0}$, and can be further removed to get a branched surface $\mathcal{B}_{S_0}'$. Moreover, by~\cite{lyu2025persistent}, Lemma 6.7, $\mathcal{B}_{S_0}'$ fully carries a lamination if and only if $\mathcal{B}_{S_0}$ does. We henceforth proceed with $\mathcal{B}_{S_0}'$ to find laminations (since its branch locus is simpler).

    The rest of the proof is parallel to what we did in Proposition~\ref{prop:asso_primi}, see Figure~\ref{fig:modified_primi_red} as an example. In fact, there is still an annulus sector on $\Sigma$, whose core curve we denote as $\delta$. We can still formally cut $\Sigma$ along $\delta$ and collapse the boundary circles to get a 4-pointed sphere $(S,z,w,\delta_z,\delta_w)$. We can still define a ``modified'' $\beta$-sink tube push\footnote{Since we reversed $S_0$ (which shares two vertices with the $\beta$-sink tube), we cannot push the whole long boundary $\beta$-arc of the sink tube; instead, we stop right before the $\beta$-sink bigon, see Figure~\ref{fig:modified_primi_red}. This is what ``modified'' means.} on $\Sigma$ that descends to $S$ (see~\cite{lyu2025persistent}, Subsection~6.3), where after the push the remaining $\beta$ on $S$ can be characterized as a rational tangle strand. We can still collapse the small circles to the basepoints\footnote{If we are working with $\mathcal{B}_{S_0}$, then when collapsing the small disk to $w$ we will actually be collapsing a small sink disk, see the shadowed region in the middle right picture of Figure~\ref{fig:modified_primi_red}. By removing the $w$-bigon and working with $\mathcal{B}_{S_0}'$ instead we can avoid this issue.}. We can still collapse $\alpha$ in its source direction to the strand $w\delta_w$, without collapsing any sink disk candidates.

    The difference here is that 1) the $\beta$-rational tangle strand is now connecting $w$ (instead of $z$) to one of $\delta_z,\delta_w$ (i.e. we are taking the other strand of the same rational tangle), and 2) when pushing the $\beta$-arcs on $S$, we now need to fix the 2 boundary $\beta$-arcs of $S_0$, one collapsing to $w$, and the other to the intersection point $e$ on $w\delta_w$ next to $w$. This is exactly the same situation as in the proof of~\cite{lyu2025persistent}, Lemma~6.17 in Subsection~6.4, where sink tube pushes are again performed recursively to eliminate sink disks.

    Similarly to the proof of Proposition~\ref{prop:asso_primi}, the branch sector corresponding to $\delta$ remains an annulus sector and is not a sink disk. It follows that after the splittings we get a sink-disk-free branched surface, and by Lemma~\ref{lem:sk_disk_free} it fully carries a lamination. Hence $\mathcal{B}_{S_0}'$ and $\mathcal{B}_{S_0}$ also fully carry laminations.
\end{proof}

\bibliography{references.bib}
\bibliographystyle{alpha}

\end{document}